\documentclass[final,leqno]{amsart}

\usepackage{lineno}
\usepackage{stmaryrd}
\usepackage{mathrsfs}

\usepackage[dvips]{graphicx}
\usepackage{amsmath,amsfonts,amssymb,latexsym}
\usepackage{graphicx,float,epsfig}
\usepackage[dvips]{color}
\usepackage[dvipsnames]{xcolor}

\newtheorem{remark}{\textit{Remark}}[section]
\newtheorem{problem}{Problem}[section]
\newtheorem{lemma}{Lemma}[section]

\newtheorem{theorem}{Theorem}[section]
\newtheorem{corollary}{Corollary}[section]

\usepackage{graphicx,float,epsfig,color,fancyhdr}
\usepackage{colortbl}
\usepackage[colorlinks=false, linkcolor=blue,  citecolor=OliveGreen,  pdfstartview={}{}]{hyperref}

\def\bu{\mathbf{u}}
\def\bv{\mathbf{v}}
\def\bw{\mathbf{w}}
\def\bf{\mathbf{f}}

\def\bpi{\boldsymbol{\Pi}}

\def\bV{{\H_0(\div;\O)}}
\def\bVV{\boldsymbol{\H}}
\def\bxi{\boldsymbol{\xi}}

\def\CT{\mathcal{T}}
\def\curl{\mathop{\mathbf{curl}}\nolimits}

\def\dim{\mathop{\mathrm{\,dim}}\nolimits}
\def\disp{\displaystyle}

\def\div{\mathop{\mathrm{div}}\nolimits}

\def\E{\mathrm{K}}

\def\G{\Gamma}

\def\H{\mathrm{H}}

\def\HutO{{\H^{1+t}(\O)}}

\def\HsO{{\H^{s}(\O)}}

\def\l{\lambda}
\def\L{\mathrm{L}}
\def\LO{{\mathrm{L}^2(\O)}}

\def\N{{\mathbb{N}}}

\def\O{\Omega}

\def\Q{{\mathrm{Q}}}
\def\R{{\mathbb{R}}}
\def\rot{\mathop{\mathrm{rot}}\nolimits}

\def\T{{\mathcal T}}

\def\btau{\boldsymbol{\tau}}
\def\mQ{\L^2(\O)}
\def\mQQ{\boldsymbol{\Q}}
\renewcommand\H{\mathrm{H}}
\newcommand{\norm}[1]{\left\|#1\right\|}

\newcommand\HdivO{{\H(\div;\O)}}
\newcommand\0{\boldsymbol{0}}
\newcommand\bn{\boldsymbol{n}}

\newcommand\bK{\boldsymbol{\mathcal{K}}}
\newcommand\bbP{\mathbb{P}}
\newcommand{\jump}[1]{\Big\llbracket #1 \Big\rrbracket}

\begin{document}


\title{A posteriori virtual element method for the acoustic vibration problem}

\author{F. Lepe}
\address{GIMNAP-Departamento de Matem\'atica,
Universidad del B\'io-B\'io, Casilla 5-C, Concepci\'on, Chile.}
\email{flepe@ubiobio.cl}
\thanks{The first author has been partially supported by DICREA through project 2120173 GI/C
Universidad del B\'io-B\'io and ANID-Chile through FONDECYT project 11200529, Chile.}
\author{D. Mora}
\address{GIMNAP-Departamento de Matem\'atica, Universidad del B\'io-B\'io,
Casilla 5-C, Concepci\'on, Chile and CI$^2$MA, Universidad de Concepci\'on, Concepci\'on, Chile.}
\email{dmora@ubiobio.cl}
\thanks{The second author was partially supported by
DICREA through project 2120173 GI/C Universidad del B\'io-B\'io,
by the National Agency for Research and Development, ANID-Chile through FONDECYT project 1220881,
by project {\sc Anillo of Computational Mathematics for Desalination Processes} ACT210087,
and by project project Centro de Modelamiento Matem\'atico (CMM), ACE210010 and FB210005,
BASAL funds for centers of excellence.}
\author{G. Rivera}
\address{Departamento de Ciencias Exactas,
Universidad de Los Lagos, Casilla 933, Osorno, Chile.}
\email{gonzalo.rivera@ulagos.cl}
\thanks{The third author was supported by through project R02/21 Universidad de Los Lagos.}
\author{I. Vel\'asquez } 
\address{Departamento de Ciencias B\'asicas, Universidad del Sin\'u El\'ias Bechara Zain\'um, Monter\'ia, Colombia.}
\email{ivanvelasquez@unisinu.edu.co}




\begin{abstract}
In two dimensions, we propose and analyze an a posteriori error estimator for the acoustic spectral problem
based on the virtual element method in $\H(\div;\O)$. Introducing an auxiliary unknown,
we use the fact that the primal formulation of the acoustic problem
is equivalent to a mixed formulation, in order to prove a superconvergence result, necessary to
despise high order terms. Under the virtual element approach, we prove that our
local indicator is reliable and globally efficient in the $\L^2$-norm. We provide numerical
results to assess the performance of the proposed error estimator.

\end{abstract}

\maketitle

\section{Introduction}
\label{SEC:INTR}

One of the most important subjects in the development of numerical methods for partial differential equations is the a posteriori error analysis, since it allows dealing with singular solutions that arise due, for instance, geometrical features of the domain or some particular boundary conditions, among others. In this sense, and in particular for eigenvalue problems arising from problems related to solid and fluid mechanics and electromagnetism, just to mention some possible applications, the a posteriori analysis has taken relevance in  recent years.
(see \cite{ADV2000,BGGG17,BGRS2,BDGG,DGP99,DPR03,MR2019,MRR_apost,Verfurth,Verfurth2} and the references therein).


The virtual element method (VEM), introduced in \cite{BBCMMR2013}, has shown remarkable results in different problems, and particularly for solving eigenproblems, showing great accuracy and flexibility in the approximation of eigenvalues and eigenfunctions. We mention \cite{ CGMMV,DV_camwa2022,GMV, GV, LR21,LR2,LMRV, MWMIMA2022,MM, MZM, MR2019,MRR1,MV-SISC} as recent works on this topic.

The acoustic vibration problem appears in important applications in engineering.
In fact, it can be used to design of structures and devices for noise reduction in aircraft or cars
mainly related with solid-structure interaction problems, among others important applications.
In the last years, several numerical methods have been developed
in order to approximate the eigenpairs of the associated spectral problem.
In particular, a virtual element discretization has been proposed in \cite{BeiraoVEMAcoustic2017}.
It is well known that one of the most important features of the virtual element method
is the efficient computational implementation and the flexibility on the geometries for meshes,
where precisely adaptivity strategies can be implemented in an easy way.
In fact, the hanging nodes that appear in the refinement of some element of the mesh,
can be treated as new nodes since adjacent non matching element interfaces are acceptable
in the VEM. Recent research papers report  interesting advantages  of the VEM in the a posteriori error analysis and adaptivity for source problems.  We refer to \cite{BM2015, CGPS, CM2019, MS} and the references therein, for instance, for a further discussion. On the other hand, a posteriori error analysis for eigenproblems by VEM
have been recently introduced in \cite{WMWMIMA2022,MR2019,MRR_apost}, where primal formulations
in $\H^1$ have been considered.


 The contribution of our work is the design and analysis of an a posteriori error estimator for the acoustic spectral problem, by means of a VEM method. The VEM that we consider in our analysis is the one introduced in \cite{BeiraoVEMAcoustic2017} for the a priori error analysis of the acoustic eigenproblem.
 We stress that the VEM method presented in \cite{BeiraoVEMAcoustic2017}
 may be preferable to more standard finite elements even in the case of triangular meshes in terms of dofs (cf. \cite[Remark 3]{BeiraoVEMAcoustic2017}).
 The formulation for the acoustic problem is written only in terms of the displacement of the fluid, which leads to a bilinear form with divergence terms, implying that the analysis for the a posteriori error indicator is not straightforward. This difficulty produced by the $\H(\div)$ formulations leads to analyze, in first place, an equivalent mixed formulation which provides suitable results in order  to control the so-called high order terms that naturally appear. This analysis depending on an equivalent mixed formulation has been previously considered in \cite{BGRS1, BGRS2} for the a posteriori analysis for the Maxwell's eigenvalue problem, inspired by the superconvergence results of \cite{LX2012} for mixed spectral formulations. 
 We will follow the same techniques for the present $\H(\div)$ framework.
 However, due to the nature of the VEM, the local indicator that we present
 contains an extra term depending on the virtual projector which needs to be analyzed carefully.

 
 The organization of our paper is the following: in section \ref{sec:model} we present the acoustic problem and the mixed equivalent formulation for it. We recall some properties of the spectrum of the spectral  problem and regularity results. In section \ref{SEC:Discrete} we found the core of the analysis of our paper, where we introduce the virtual element method for our spectral problem and technical results that will be needed to establish a superconvegence result, with the aid of mixed formulations. Section \ref{sec:apost} is dedicated to the a posteriori error analysis, where we introduce our local and global indicators which, as is customary in the posteriori error analysis, will be reliable and efficient. In section \ref{sec:numerics}, we report numerical tests where we assess the performance of our estimator.  We end the article with some concluding remarks.
 
Throughout this work, $\O$ is a generic Lipschitz bounded domain of $\R^2$. For $s\geq 0$,
$\norm{\cdot}_{s,\O}$ stands indistinctly for the norm of the Hilbertian
Sobolev spaces $\HsO$ or $[\HsO]^2$ with the convention
$\H^0(\O):=\LO$. We also define the Hilbert space
$\HdivO:=\{\btau\in[\LO]^2:\ \div\btau\in\LO\}$, whose norm
is given by $\norm{\btau}^2_{\div,\O} 
:=\norm{\btau}_{0,\O}^2+\|\div\btau\|^2_{0,\O}$. For
$s\geq 0$, we define  the Hilbert space 
$\H^{s}(\div;\O):=\{\btau\in[\H^s(\O)]^2:\ \div\btau\in\HsO\}$, whose norm
is given by $\|\btau\|^2_{\H^s(\div;\O)}
:=\|\btau\|_{s,\O}^2+\|\div\btau\|^2_{s,\O}$. Finally,
we employ $\0$ to denote a generic null vector and
the relation $\texttt{a} \lesssim \texttt{b}$ indicates that $\texttt{a} \leq C \texttt{b}$, with a positive constant $C$ which is independent of $\texttt{a}$, $\texttt{b}$, and the size of the elements in the mesh. The value of $C$ might change at each occurrence. We remark that we will write the constant $C$ only when is needed.

\section{The spectral problem}
\label{sec:model}

We consider the free vibration problem for an acoustic fluid within
a bounded rigid cavity $\O\subset\R^2$ with polygonal boundary $\G$
and  outward unit normal vector $\bn$:
\begin{equation}
\label{def:acoustic_problem}
\left\{\begin{array}{ll}
\vspace{0.1cm}
-\omega^2 \varrho\bw=-\nabla p\quad &\text{in }\O, \\
\vspace{0.1cm}
p=-\varrho c^2\div \bw \quad &\text{in }\O,\\
\bw\cdot\bn=0 \quad &\text{on }\G,
\end{array}\right.
\end{equation}
where $\bw$ is the fluid displacement, $p$ is the pressure fluctuation,
$\varrho$ the density, $c$ the acoustic speed and $\omega$
the vibration frequency. For simplicity on the forthcoming analysis,
we consider $\varrho$ and $c$ equal to one.

Multiplying the first equation  in \eqref{def:acoustic_problem} by a test function $\btau\in\bV$, where 
$$\bV:=\left\{\btau\in\HdivO: \btau\cdot\bn=0\quad \text{on }\G\right\},$$ 
integrating by parts, using the boundary condition
and eliminating the pressure $p$, we arrive at the following
weak formulation 
\begin{problem}
\label{P1}
 Find $(\l,\bw)\in \R\times\bV$, $\bw\neq 0$, such that 
\begin{equation*}
\label{2}
\int_{\O}  \div\bw\div \btau =\l \int_{\O}\bw\cdot\btau \qquad \forall\, \btau\in\bV,
\end{equation*}
\end{problem}
\noindent where  $\l:=\omega^2$. 
It is well known that the spectrum of Problem \ref{P1} consists in   a sequence of eigenvalues    $\left\{0\right\}\cup\left\{\l_k\right\}_{k\in\N}$, such that
\begin{enumerate}
\item[\textit{i)}] $\l=0$ is an infinite-multiplicity eigenvalue
and its associated eigenspace is $\H_{0}(\div^{0};\O):=\{\btau\in\bV\,:\,\, \div\btau=0\,\,{\text{in}}\,\, \O\}$;
\item[\textit{ii)}] $\left\{\l_k\right\}_{k\in\N}$ is a
sequence of finite-multiplicity eigenvalues  which satisfy $\l_{k}\rightarrow\infty$.
\end{enumerate}
%
%

To perform an a posteriori error analysis for spectral problems,  we need the so called \emph{superconvergence result}, in order to neglect high order terms as has been proved in \cite{LX2012} and already applied in, for instance, the Maxwell's eigenvalue problem \cite{BGRS1, BGRS2}. In order to obtain this superconvergence result, we begin by introducing an equivalent mixed formulation for Problem \ref{P1}. For $\l\neq 0$ let us introduce the 
unknown
\begin{equation}
\label{eq:u}
u:=-\frac{\div\bw}{\lambda}\in \mQ.
\end{equation}
To remain consistent with the notations,
we will denote by $(\cdot,\cdot)_{0,\O}$ the $\L^2(\O)$ inner-product. 

With the aid of \eqref{eq:u} we write the following mixed eigenproblem:
\begin{problem}
\label{mixedP1}
Find $(\lambda,\bw,u)\in\mathbb{R}\times\bV\times\mQ$, with $(\bw,u)\neq\boldsymbol{0}$, such that
\begin{equation*}
\left\{
\begin{array}{rcll}
\disp\int_{\O}\bw\cdot\btau+\int_{\O}u\div \btau&=&0\quad&\forall \,\btau\in \bV,\\
\disp\int_{\O}v\div\bw&=&\disp-\lambda\int_{\O} uv\quad&\forall \,v\in \mQ. 
\end{array}
\right.
\end{equation*}
\end{problem}
It is easy to check that the spectral Problem \ref{P1} and \ref{mixedP1} are equivalent, except for $\l=0$ on the following sense:
\begin{itemize}
\item If $(\lambda,\bw)$ is a solution of Problem~\ref{P1}, with $\lambda\neq 0$, then $(\lambda,\bw,-\div\,\bw/\lambda)$ is solution of Problem \ref{mixedP1}.
\item If  $(\lambda,\bw, u)$ is a solution of Problem \ref{mixedP1}, then $(\lambda,\bw)$ is solution of Problem \ref{P1} and $u$ is defined as in \eqref{eq:u}.
\end{itemize}

We introduce the  bounded and symmetric bilinear
forms $a:\bV\times\bV\rightarrow\mathbb{R}$
and $b:\bV\times \mQ\rightarrow\mathbb{R}$,
defined by
\begin{align*}
a(\bw,\btau) :=\int_{\Omega}\bw\cdot\btau, \quad \bw,\btau\in\bV,
\quad
b(\btau,v) :=\int_{\O}v\div\btau, 
\qquad \btau\in\bV, \,v\in \mQ,
\end{align*}
which allows us to  we rewrite Problem~\ref{mixedP1} as follows:
 \begin{problem}
\label{P2}
Find $(\l,\bw,u)\in\R\times\bV\times \mQ$, $(\bw,u)\ne(\boldsymbol{0},0)$, such that
\begin{equation*}
\left\{
\begin{array}{rcll}
a(\bw,\btau)+b(\btau, u)&=&0 \qquad\qquad&\forall \btau\in\bV,\\
b(\bw,v)&=&-\lambda(u,v)_{0,\O}\qquad&\forall v\in \mQ.
\end{array}
\right.
\end{equation*}
\end{problem}
 \begin{remark}
 \label{rm:1}
 It is easy to check that if $(\l,\bw,\bu)$ is a solution of Problem \ref{P2}, then
 $$\bw=\nabla u\qquad\text{and}\qquad \div \bw=-\lambda u.$$
 \end{remark}
 Let $\mathcal{K}$ be the kernel of bilinear form $b(\cdot,\cdot)$ defined by:
 \begin{equation*}
 \mathcal{K}:=\{\btau\in\bV\,:\,\, \div\btau=0\,\,{\text{in}}\,\, \O\}.
 \end{equation*}

 It is well-known that bilinear form $a(\cdot,\cdot)$
 is elliptic in $\mathcal{K}$ and that $b(\cdot,\cdot)$
 satisfies the following inf-sup condition (see \cite{bbf-2013})
 \begin{equation}
 \label{cont-infsup}
 \displaystyle\sup_{\boldsymbol{0}\neq\btau\in\bV}
 \frac{b(\btau,v)}{\|\btau\|_{\div,\O}}\geq\beta\|v\|_{0,\O}\qquad\forall v\in \mQ, 
 \end{equation}
where $\beta$ is a positive constant.
\begin{remark}
The eigenvalues of Problem~\ref{P2} are positive.
Indeed, taking $\btau=\bw$ and $v=u$ in Problem~\ref{P2} and
subtracting the resulting forms, we  obtain
$$\lambda=\frac{a(\bw,\bw)}{\|u\|_{0,\O}^2}\geq0.$$
In addition, $\lambda=0$ implies $(\bw,u)=(\boldsymbol{0},0)$. 
\end{remark}

Let us introduce the following source problem:
For a given $g\in\mQ$, the pair $(\widetilde{\bw},\widetilde{u})\in \bV\times\mQ$
is the solution of the following well posed problem
\begin{align}
a(\widetilde{\bw},\btau)+b(\btau,\widetilde{u})&=0\quad\forall\,\btau\in\bV,\label{eq:mixST1}\\
b( \widetilde{\bw}, v)&=-(g,v)_{0,\O}\quad\forall \,v\in\mQ.\label{eq:mixST2}
\end{align}



According to \cite{agmon}, the regularity for the solution of system \eqref{eq:mixST1}--\eqref{eq:mixST2},
(the associated source problem to Problem \ref{P2}) is the following: there exists a constant
$\widetilde{r}>1/2$ depending on $\Omega$ such that the solution $\widetilde{u}\in\H^{1+\widetilde{r}}(\Omega)$,
where $\widetilde{r}$ is at least 1 if $\Omega$ is convex and $\widetilde{r}$ is at least
$\pi/\omega-\varepsilon$, for any $\varepsilon>0$ for a non-convex domain,
with $\omega<2\pi$ being the largest reentrant angle of $\Omega$. Hence we have the following well known additional regularity result for the source problem \eqref{eq:mixST1}--\eqref{eq:mixST2}.
\begin{equation}
\label{eq:additional_source}
\|\widetilde{\bw}\|_{\widetilde{r},\O}+\|\widetilde{u}\|_{1+\widetilde{r},\O}\lesssim\|g\|_{0,\O}.
\end{equation}

Also, the eigenvalues are well characterized for this problem as is stated in the following result (see \cite{BO} for instance).
\begin{lemma}
\label{lmm:charact}
The eigenvalues of Problem \ref{mixedP1} consist in a sequence of positive eigenvalues  $\{\lambda_n\,:\,n\in\mathbb{N}\}$,
such that $\lambda_n\rightarrow\infty$ as $n\rightarrow \infty$.
In addition, the following additional regularity result holds true for eigenfunctions
\begin{equation*}
\label{eq:addi_eigen}
\|\bw\|_{r,\O}+\|\div\bw\|_{1+r,\O}+\|u\|_{1+r,\O}\lesssim\|u\|_{0,\O},
\end{equation*}
with $r>1/2$ and  the hidden constant depending on the eigenvalue.
\end{lemma}

\section{The virtual element discretization}
\label{SEC:Discrete}

We begin this section recalling the mesh construction and the
assumptions considered to introduce the discrete virtual element
space. Then, we will introduce a virtual element discretization of
Problem~\ref{P1}  and provide a spectral characterization
of the resulting discrete eigenvalue problem.

Let $\left\{\CT_h\right\}_h$ be a sequence of decompositions of $\O$
into polygons $\E$. Let $h_\E$ denote the diameter of the element $\E$
and $h:=\displaystyle\max_{\E\in\O}h_\E$. For the analysis, the following
standard assumptions on the meshes are considered (see \cite{BBMR2015,ultimo}):
there exists a positive real number $C_{\T}$ such that,
for every  $\E\in \T_h$ and for every $h$.

\begin{itemize}
\item[$\mathbf{A_1}$:]  the ratio between the shortest edge
and the diameter $h_\E$ of $\E$ is larger than $C_{\T}$,
\item[$\mathbf{A_2}$:]  $\E\in\CT_h$ is star-shaped with
respect to every point of a  ball
of radius $C_{\T}h_\E$.
\end{itemize}

For any subset $S\subseteq\R^2$ and nonnegative
integer $k$, we indicate by $\bbP_{k}(S)$ the space of
polynomials of degree up to $k$ defined on $S$.
To keep the notation simpler, we denote by $\bn$ a general normal
unit vector, its precise definition will be clear from the context.
We consider now a  polygon $\E$ and  define the following
local finite dimensional space for $k\ge0$ (see \cite{ultimo,BBMR2015}):
\begin{equation*}
\bVV_{h}^{\E}:=\left\{\btau_h\in\H(\div;\E)\cap\H(\rot;\E):\left(\btau_h\cdot\bn\right)\in
\bbP_{k}(\ell)\,
\forall \ell\in\partial \E,\, \div\btau_h\in \bbP_{k}(\E),\,
\rot\btau_h=0\text{ on }\E\right\},
\end{equation*}

We define the following degrees of freedom
for functions $\btau_h$ in $\bVV_{h}^{\E}$:
 \begin{align}
 \label{freedom}
\displaystyle \int_{\ell}\left(\btau_h\cdot\bn\right) q\; ds&\quad \forall q\in \bbP_{k}(\ell)\quad \forall \text{ 
edge }\ell\in \partial \E,\\
\label{freedom2}
\displaystyle \int_{\E}\btau_h\cdot\nabla q&\quad \forall q\in \bbP_{k}(\E)/\bbP_{0}(\E),
\end{align}
which are unisolvent (see \cite[Proposition 1]{BeiraoVEMAcoustic2017}).

For every decomposition $\CT_h$ of $\O$ into polygons $\E$, we define
\begin{align}
\bVV_h:=\left\{\btau_h\in\bV:\btau_h|_\E\in \bVV_h^\E\right\}.\nonumber
\end{align}
In agreement with the local choice  we choose the following degrees of freedom:
\begin{align*}
\label{globalfreedom}
\displaystyle \int_{\ell}\left(\btau_h\cdot\bn\right)q\; ds&\quad \forall q\in \bbP_{k}(\ell)\quad \text{ for all internal edges
 }\ell\in \CT_h,\\
\displaystyle \int_{\E}\btau_h\cdot\nabla q\;&\quad \forall q\in \bbP_{k}(\E)/\bbP_{0}(\E)\quad \text{ in each element }\E\in \CT_h.\\
\end{align*}
In order to construct the discrete scheme, we need some preliminary
definitions. For each element $\E\in\CT_h$, we define the space
\begin{equation*}\label{Ve}
 \widehat{\bVV}_h^\E:=\nabla(\bbP_{k+1}(\E))\subset \bVV_{h}^{\E}.
\end{equation*}
Next, we define the orthogonal projector
$\bpi_h^\E:[\L^2(\E)]^2\longrightarrow\widehat{\bVV}_h^\E$ by
\begin{equation}
\label{numero}
\int_{\E}\bpi_h^\E\btau\cdot\widehat{\bu}_h=\int_{\E}\btau\cdot\widehat{\bu}_h
\qquad \forall \widehat{\bu}_h\in \widehat{\bVV}_h^\E,
\end{equation}
and we point out that $\bpi_h^\E\btau_h$ is explicitly computable
for every $\btau_h\in\bVV_h^\E$ using only its degrees of freedom
\eqref{freedom}--\eqref{freedom2}.
In fact, it is easy to check that, for all $\btau_h\in\bVV_h^\E$ and
for all $ q\in \bbP_{k+1}(\E)$,
\begin{equation*}
\int_{\E}\bpi_h^\E\btau_h\cdot\nabla q=\int_{\E}\btau_h\cdot\nabla q
=-\int_{\E}q\div\btau_h+\int_{\partial \E}\left(\btau_h\cdot\bn\right)q\;ds.
\end{equation*}


On the other hand, let $S^\E(\cdot,\cdot)$ be any symmetric positive
definite (and computable) bilinear form that satisfies
\begin{equation}
\label{eq:20}
c_0\,\int_{\E}\btau_h\cdot\btau_h\leq S^{\E}(\btau_h,\btau_h)\leq c_1\,\int_{\E}\btau_h\cdot\btau_h
\qquad\forall \,\btau_h\in \bVV_h^\E,
\end{equation}
for some positive constants $c_0$ and $c_1$ depending only on the shape regularity constant
$C_{\T}$ from mesh assumptions $\mathbf{A_1}$ and $\mathbf{A_2}$. Then, we define on each $\E$ the following bilinear form:
\begin{equation*}
\label{21}
a_h^{\E}(\bu_h,\btau_h)
:=\int_{\E}\bpi_h^\E\bu_h\cdot\bpi_h^\E\btau_h
+S^{\E}\big(\bu_h-\bpi_h^\E \bu_h,\btau_h-\bpi_h^\E \btau_h\big)
\qquad \bu_h,\btau_h\in\bVV_h^\E, 
\end{equation*}
and, in a natural way,
 $$
a_h(\bu_h,\btau_h)
:=\sum_{\E\in\CT_h}a_h^{\E}(\bu_h,\btau_h),
\qquad \bu_h,\btau_h\in\bVV_h.
$$
The following properties of the bilinear form
$a_h^\E(\cdot,\cdot)$ are easily derived
(by repeating, in our case, the arguments from  \cite[Proposition~4.1]{ultimo}).
\begin{itemize}
\item \textit{Consistency}: 
\begin{equation*}
\label{consistencia1}
a_h^{\E}(\widehat{\bu}_h,\btau_h)
=\int_{\E}\widehat{\bu}_h\cdot\btau_h
\qquad\forall \widehat{\bu}_h\in\widehat{\bVV}_h^\E\quad\forall \btau_h\in\bVV_h^\E,\quad\forall \,\E\in\CT_h.
 \end{equation*}
\item \textit{Stability}: There exist two positive constants $\alpha_*$
and $\alpha^*$, independent of $\E$, such that:
\begin{equation*}
\label{consistencia2}
\alpha_*\int_{\E}\btau_h\cdot\btau_h
\leq a_h^{\E}(\btau_h,\btau_h)
\leq\alpha^*\int_{\E}\btau_h\cdot\btau_h
\qquad\forall\, \btau_h\in\bVV_h^\E,\quad \forall \,\E\in\CT_h. 
\end{equation*}
\end{itemize}

Now, we are in position to write the virtual
element discretization of Problem~\ref{P1}.
\begin{problem}
\label{P3}
Find $(\l_h,\bw_h)\in \R\times\bVV_h$, $\bw_h\neq 0$ such that 
\begin{equation*}\label{vp}
(\div\bw_h,\div\btau_h)_{0,\O}=\l_h a_h(\bw_h,\btau_h) \qquad \forall \,\btau_h\in\bVV_h.
\end{equation*}
\end{problem}
We have the following spectral characterization
of the discrete eigenvalue Problem~\ref{P3} (see \cite{BeiraoVEMAcoustic2017}).

\begin{remark}
\label{remark2}
There exist $M_h:=\dim(\bVV_h)$ eigenvalues of Problem \ref{P3}
repeated according to their respective multiplicities,
which are $\left\{0\right\}\cup\left\{\l_{hk}\right\}_{k=1}^{N_h}$, where:
\begin{enumerate}
\item[\textit{i)}] the eigenspace associated with $\l_h=0$ is $\bK_h:=\{\bv_h\in \bVV_h: \div \bv_{h}=0\}$;
\item[\textit{ii)}] $\l_{hk}>0$,
$k=1,\dots,N_h:=M_h-\dim(\bK_h)$, are non-defective eigenvalues repeated
according to their respective multiplicities.
\end{enumerate}
\end{remark}

%
Now, we introduce  the virtual
element discretization of Problem~\ref{P2}.
\begin{problem}
\label{mixted:discretP}
Find $(\l_h,\bw_h,u_{h})\in \R\times\bVV_h\times\mQQ_h$, $(\bw_h,u_{h})\neq (\boldsymbol{0},0)$, such that 
\begin{equation*}
\left\{
\begin{array}{rcll}
a_{h}(\bw_h,\btau_h)+b(\btau_{h},u_{h})&=&0\quad&\forall\,\btau_h\in\bVV_h, \\
b(\bw_{h},v_h)&=&-\l_h(u_{h},v_h)_{0,\O}\quad&\forall \,v_h\in\mQQ_{h}, 
\end{array}
\right.
\end{equation*}
where $\mQQ_{h}:=\left\{q\in\LO: q|_\E\in\mathbb{P}_k(\E)\quad \forall \,\E\in \CT_h\right\}, k\ge0.$
\end{problem}
%
%

We also  introduce the $\L^2(\O)$-orthogonal projection 
$$P_k:\L^2(\O)\rightarrow\mQQ_{h},$$
and  the following approximation result (see \cite{BBMR2015}): if $0\leq s \leq k+1$,  it holds
\begin{equation}
\label{eq:salim}
\|v-P_kv\|_{0,\O}\lesssim h^{s}\|v\|_{s,\O}\qquad\forall v\in\H^s(\O).
\end{equation}

The next  two technical results establish the approximation
properties for $\btau_I$ and their proofs  can be found in 
\cite[Appendix]{BeiraoVEMAcoustic2017}.
\begin{lemma}
\label{lemmainter}
Let $\btau\in\bV$ be such that $\btau\in[\H^{t}(\O)]^2$ with $t>1/2$.
There exists $\btau_I\in\bVV_h$ that satisfies:
$$\div\btau_I=P_{k}(\div \btau)\quad\text{ in }\O.$$
Consequently, for all $\E\in \CT_h$
\begin{equation*}
\|\div\btau_I\|_{0,\E}\leq \|\div \btau\|_{0,\E},
\end{equation*}
and, if $\div\btau|_{\E}\in\H^\delta(\E)$ with $\delta\geq 0$, then
\begin{equation*}
\|\div \btau-\div\btau_I\|_{0,\E}\lesssim h_\E^{\min\{\delta,k+1\}}|\div \btau|_{\delta,\E}. 
\end{equation*}
\end{lemma}
\begin{lemma}
\label{lemmainterV_I}
 Let $\btau\in\bV$ be such that $\btau\in[\H^{t}(\O)]^2$ with $ t >1/2$. Then,
there exists $\btau_I\in\bVV_h$ such that,
 if $1\leq t\leq k+1$, there holds
\begin{equation*}
\|\btau-\btau_I\|_{0,\E}\lesssim h_\E^{t}|\btau|_{t,\E},
\end{equation*}
 where the hidden constant is independent of $h$. Moreover, if $1/2 <t \leq1$, then
\begin{equation*}
\|\btau-\btau_I\|_{0,\E}\lesssim h_\E^{t}|\btau|_{t,\E}+h_\E\|\div\btau\|_{0,\E}.
\end{equation*}
\end{lemma}
%
%

Let $\bpi_h$ be defined in $[\L^2(\O)]^2$ by $(\bpi_h \btau)|_\E:=\bpi_h^\E\btau\text{ for all }\E\in \T_h$,
 where $\bpi_h^\E$ is the operator defined in \eqref{numero}, and that
 satisfies the following result proved in \cite[Lemma~8]{BeiraoVEMAcoustic2017}.
 \begin{lemma}
 \label{A}
For  every $q\in\HutO$ with $1/2< t\le k+1$, there holds 
\begin{align*}
&\quad \left\|\nabla  q-\bpi_h (\nabla  q)\right\|_{0,\O}\lesssim h^{t}\Vert\nabla q\Vert_{t,\O}.
\end{align*}
\end{lemma}


As a consequence of the previous result we have the following
estimate.
\begin{lemma}
\label{lm:A1}
For all $r>\frac12$ as in Lemma~\ref{lmm:charact}, the  following error estimate holds
$$\|\bw-\bpi_h\bw\|_{0,\O}\lesssim h^{\min\{r,k+1\}}.$$
\end{lemma}
\begin{proof}
The proof follows directly from Remark \ref{rm:1}, Lemma~\ref{lmm:charact} and Lemma \ref{A}.
\end{proof}

The following results gives us the error estimates between
the eigenfunctions and eigenvalues of Problems \ref{mixted:discretP}
and \ref{P2}.
\begin{theorem}
\label{cotadoblepandeo} 
For all $r>\frac12$ as in Lemma \ref{lmm:charact}, the  following error estimates hold 
\begin{equation}\label{eq:l_lh}
\begin{split}
\|\bw-\bw_{h}\|_{\div,\O}+\left\|u-u_h\right\|_{0,\O}&\lesssim h^{\min\{r,k+1\}},\\ 
\left|\l-\l_h\right|&\lesssim  h^{2\min\{r,k+1\}},
\end{split}
\end{equation}
 where the hidden constants are independent of $h$.
\end{theorem}
\begin{proof}
The proof follows by repeating the arguments in \cite[Theorems~4.2, 4.3 and 4.4]{LR2}.
\end{proof}

For the a  posteriori error analysis that will be developed in Section~\ref{sec:apost},
we will need the following auxiliary results, which have been adapted from \cite{BGRS1,BGRS2}.

In what follows,  let $(\l,\bw,u)$ be a solution of Problem~\ref{P2},
where we assume that $\l$ is a simple eigenvalue. 
Let $(\bw,u)$ be an associated eigenfunction which we normalize
by taking $\|u\|_{0,\O}=1$. Then, for each mesh $\CT_{h}$,
there exists a solution $(\l_{h},\bw_{h},u_{h})$ of Problem~\ref{mixted:discretP}
such that $\l_{h}$ converges to $\l$, as $h$ goes to zero, $\|u_{h}\|_{0,\O}=1$,
and Theorem~\ref{cotadoblepandeo} holds true.


Let us introduce the following well posed source problem with data $(\l,u)$:
Find $(\widehat{\bw}_h,\widehat{u}_h)\in\bVV_h\times\mQQ_h$, such that
\begin{equation}\label{eq:3.3a}
\left\{
\begin{array}{rcll}
a_h(\widehat{\bw}_h,\btau_h)+b(\btau_h,\widehat{u}_h)&=&0\quad&\forall\,\btau_h\in\bVV_h, \\
b(\widehat{\bw}_h,v_h)&=&-\l(u,v_h)_{0,\O}\quad&\forall \,v_h\in\mQQ_{h}.
\end{array}
\right.
\end{equation}




With this mixed problem at hand, we will prove the following
technical lemmas with the goal of derive a superconvergence
result for our VEM. To make matters precise, the forthcoming
analysis is inspired by \cite{BGRS2}, where the authors have
generalized the results previously obtained by \cite{bbf-2013, DR, gardini}.
We begin proving a higher-order approximation between  $\widehat{u}_{h}$ and $P_{k}u$.
\begin{lemma}
\label{lmm:rodolfo9}
Let $(\lambda, \bw,u)$ be a solution of Problem~\ref{P2}
and $(\widehat{\bw}_{h},\widehat{u}_{h})$ be a solution of the mixed formulation
\eqref{eq:3.3a}. Then, there holds 
\begin{equation*}
\|\widehat{u}_h-P_ku\|_{0,\O}\lesssim h^{\widetilde{r}}\left(\|\bw-\widehat{\bw}_h\|_{\div,\O}
+\|\bw-\bpi_h\bw\|_{0,\O}\right),
\end{equation*}
where $\widetilde{r}\in (\frac{1}{2},1]$ and the hidden constant is independent of $h$.
\end{lemma}
\begin{proof}
Let $(\widetilde{\bw}, \widetilde{u})\in\bV\times\mQ$ be the unique solution
of the following well posed mixed problem
\begin{equation}\label{eq:mixed_tildes}
\left\{
\begin{array}{rcll}
a(\widetilde{\bw},\btau)+b(\btau,\widetilde{u})&=&0\quad&\forall\,\btau\in\bV, \\
b(\widetilde{\bw},v)&=&-(\widehat{u}_{h}-P_ku,v)_{0,\O}\quad&\forall \,v\in\mQ.
\end{array}
\right.
\end{equation}
Notice that \eqref{eq:mixed_tildes} is exactly problem \eqref{eq:mixST1}--\eqref{eq:mixST2}
with datum $\widehat{u}_{h}-P_ku$. Hence, since this problem is well posed,
we have the following regularity result, consequence of \eqref{eq:additional_source}
\begin{equation}
\label{eq:joven_regularity}
\|\widetilde{\bw}\|_{\widetilde{r},\O}+\|\widetilde{u}\|_{1+\widetilde{r},\O}\lesssim \|\widehat{u}_{h}-P_ku\|_{0,\O}.
\end{equation}
Observe that, thanks to the definition of $P_k$, the first equation of
Problem~\ref{P2}, and the first equation of \eqref{eq:3.3a}, we have
\begin{multline}
\label{eq:joven1}
 \|\widehat{u}_{h}-P_ku\|_{0,\O}^2=-b(\widetilde{\bw},\widehat{u}_h-P_ku)=-b(\widetilde{\bw}_{I},\widehat{u}_h-P_ku)=-b(\widetilde{\bw}_{I},\widehat{u}_h-u)\\
 =-b(\widetilde{\bw}_I,\widehat{u}_h)+b(\widetilde{\bw}_I,u)=a_h(\widehat{\bw}_h,\widetilde{\bw}_I)-a(\bw,\widetilde{\bw}_I).
\end{multline}
In the last two terms of \eqref{eq:joven1}, we add and subtract $\bpi_h\bw$ in order to obtain
\begin{align*}
a_h(\widehat{\bw}_h,\widetilde{\bw}_I)-a(\bw,\widetilde{\bw}_I)&=\underbrace{a_h(\widehat{\bw}_h-\bpi_h\bw,\widetilde{\bw}_I)+a(\bpi_h\bw-\widehat{\bw}_h,\widetilde{\bw}_I)}_{A_\mathrm{I}}+a(\widehat{\bw}_h-\bw,\widetilde{\bw}_I)\\
&=A_\mathrm{I}+\underbrace{a(\widehat{\bw}_h-\bw,\widetilde{\bw}_I-\widetilde{\bw})}_{A_{\mathrm{II}}}+a(\widehat{\bw}_h-\bw,\widetilde{\bw})\\
&=A_\mathrm{I}+A_\mathrm{II}-b(\widehat{\bw}_h-\bw,\widetilde{u})\\
&=A_\mathrm{I}+A_\mathrm{II}-b(\widehat{\bw}_h-\bw,\widetilde{u}-P_k\widetilde{u})-b(\widehat{\bw}_h-\bw,P_k\widetilde{u}),
\end{align*}
where we have used the first equation of system \eqref{eq:mixed_tildes}. Moreover, testing the second equation in Problem \ref{P2} and \eqref{eq:3.3a} with $P_k\widetilde{u}\in\mQQ_{h}$, we have
$b(\widehat{\bw}_h-\bw,P_k\widetilde{u})=0,$
and hence
\begin{equation*}
\|\widehat{u}_{h}-P_ku\|_{0,\O}^2=A_\mathrm{I}+A_\mathrm{II}-b(\widehat{\bw}_h-\bw,\widetilde{u}-P_k\widetilde{u}).
\end{equation*}

Our next task is to estimate each terms on the right
hand side of the identity above. We begin with $A_\mathrm{I}$.
\begin{align}
\nonumber A_\mathrm{I}&=a_h(\widehat{\bw}_h-\bpi_h\bw,\widetilde{\bw}_I-\bpi_h\widetilde{\bw})+a(\widehat{\bw}_h-\bpi_h\bw,\bpi\widetilde{\bw}-\widetilde{\bw}_I)\\
\nonumber &\lesssim\|\widehat{\bw}_h-\bpi_h\bw\|_{0,\O}\|\widetilde{\bw}_I-\bpi\widetilde{\bw}\|_{0,\O}\\
\nonumber &\lesssim \big(\|\widehat{\bw}_h-\bw\|_{0,\O}+\|\bw-\bpi_h\bw\|_{0,\O}\big)\big(\|\widetilde{\bw}_I-\widetilde{\bw}\|_{0,\O}+\|\widetilde{\bw}-\bpi_h\widetilde{\bw}\|_{0,\O} \big)\\
&\lesssim h^{\widetilde{r}}\big(\|\widehat{\bw}_h-\bw\|_{0,\O}+\|\bw-\bpi_h\bw\|_{0,\O} \big)\|\widehat{u}_h-P_ku\|_{0,\O},\label{eq:cota_A1}
\end{align}
where we have applied  Lemmas \ref{lemmainter} and \ref{A} 
and \eqref{eq:joven_regularity}. Now, for $A_\mathrm{II}$ we have
\begin{equation}
\label{eq:cota_A2}
A_\mathrm{II}=a(\widehat{\bw}_h-\bw,\widetilde{\bw}_I-\widetilde{\bw})\lesssim h^{\widetilde{r}}\|\bw-\widehat{\bw}_h\|_{0,\O}\|\widehat{u}_h-P_ku\|_{0,\O},
\end{equation}
and finally, invoking \eqref{eq:joven_regularity}, we obtain
\begin{equation}
\label{eq:b_cota}
b(\widehat{\bw}_h-\bw, \widetilde{u}-P_k\widetilde{u})\lesssim h^{\widetilde{r}}\|\div(\widehat{\bw}_h-\bw)\|_{0,\O}\|\widehat{u}_h-P_ku\|_{0,\O}.
\end{equation}
Collecting \eqref{eq:cota_A1}, \eqref{eq:cota_A2} and \eqref{eq:b_cota}, we have
\begin{equation*}
\|\widehat{u}_h-P_ku\|_{0,\O}\lesssim h^{\widetilde{r}}\left(\|\bw-\widehat{\bw}_h\|_{\div,\O}+\|\bw-\bpi_h\bw\|_{0,\O}\right), 
\end{equation*}
concluding the proof.
\end{proof}

The following auxiliar result shows that the term $\|\bw-\widehat{\bw}_h\|_{\div,\O}$ is bounded.
\begin{lemma}
\label{lmm:rodolfo10}
Let $(\lambda, \bw,u)$ be a solution of Problem \ref{P2}
and $(\widehat{\bw}_{h},\widehat{u}_{h})$ be a solution of \eqref{eq:3.3a}. Then, there holds
\begin{equation*}
\|\bw-\widehat{\bw}_h\|_{\div,\O}\lesssim h^r,
\end{equation*}
with $r>1/2 $ as in Lemma \ref{lmm:charact} and  the hidden constant is independent of $h$.
\end{lemma}
\begin{proof}
Let  $(\lambda_{h}, \bw_{h},u_{h})$ be  solution of Problem \ref{mixted:discretP}.
Now, subtracting \eqref{eq:3.3a} from Problem \ref{mixted:discretP}  we obtain
\begin{equation*}
\left\{
\begin{array}{rcll}
a_h(\bw_{h}-\widehat{\bw}_h,\btau_h)+b(\btau_h,u_{h}-\widehat{u}_h)&=&0\quad&\forall\,\btau_h\in\bVV_h, \\
b(\bw_{h}-\widehat{\bw}_h,v_h)&=&(\l u-\l_{h}u_{h},v_h)_{0,\O}\quad&\forall \,v_{h}\in\mQQ_{h}.
\end{array}
\right.
\end{equation*}
First, using the inf-sup condition for bilinear form $b(\cdot,\cdot)$ (cf. \eqref{cont-infsup}) we have
\begin{equation*}
\|u_{h}-\widehat{u}_h\|_{0,\O}\le\|\bw_{h}-\widehat{\bw}_{h}\|_{0,\O}.
\end{equation*}
Thus,
\begin{equation*}
\|\bw_{h}-\widehat{\bw}_{h}\|_{0,\O}\lesssim\|\l_{h}u_{h}-\l u\|_{0,\O}.
\end{equation*}
Moreover from the second equation, we get
$$\Vert\div(\bw_{h}-\widehat{\bw}_{h})\Vert_{0,\O}
\le \Vert\l_h u_h-\l u\Vert_{0,\O}.$$
On the other hand, from the triangle inequality and the above estimates, we obtain
\begin{align*}
\|\bw-\widehat{\bw}_{h}\|_{\div,\O}&\leq  \|\bw-\bw_{h}\|_{\div,\O}+\|\bw_{h}-\widehat{\bw}_{h}\|_{\div,\O}\\
&\lesssim  \|\bw-\bw_{h}\|_{\div,\O}+\|\l_{h}u_{h}-\l u\|_{0,\O}\\
&\lesssim \|\bw-\bw_{h}\|_{\div,\O}+|\l_{h}-\l |\|u_{h}\|_{0,\O}+|\l|\|u_{h}- u\|_{0,\O},
\end{align*}
where using  \eqref{eq:l_lh} we conclude the proof.
\end{proof}

We have the following essential identity to conclude the superconvergence result
presented in Lemma~\ref{lmm:rodolfo11}.

\begin{multline}\label{eq:lemma3.7}
-\lambda_h(\widehat{u}_h,u_h)=-\lambda_h(u_h,\widehat{u}_h)=b(\bw_h,\widehat{u}_h)\\
=-a_h(\widehat{\bw}_h,\bw_h)
=-a_h(\bw_h,\widehat{\bw}_h)=b(\widehat{\bw}_h,u_h)=-\lambda(u,u_h).
\end{multline}
\begin{lemma}
\label{lmm:rodolfo11}
Let $(\lambda, \bw,u)$ and $(\lambda_{h}, \bw_{h},u_{h})$ be  solutions of Problems \ref{P2} and \ref{mixted:discretP}, respectively, with $\|u\|_{0,\O}=\|u_{h}\|_{0,\O}=1$. Then, there holds
\begin{equation*}
\|P_ku-u_{h}\|_{0,\O}\lesssim h^{2\widetilde{r}},
\end{equation*}
where $\widetilde{r}\in (\frac{1}{2},1]$ and the hidden constant are independent of $h$.
\end{lemma}
\begin{proof}
Let $(\widehat{\bw}_{h},\widehat{u}_{h})$ be the solution of \eqref{eq:3.3a}. From the triangle inequality we have
\begin{equation}
\label{eq:joven-cota}
\|P_ku-u_{h}\|_{0,\O}\leq \|P_{k}u-\widehat{u}_{h}\|_{0,\O}+\|\widehat{u}_{h}-u_{h}\|_{0,\O}.
\end{equation} 
Now, adapting the arguments of \cite[Lemma 11]{BGRS2} and  using \eqref{eq:lemma3.7}, we derive the following estimate
\begin{equation*}
\label{eq:lmm_11}
\|\widehat{u}_{h}-u_{h}\|_{0,\O}^{2}\lesssim \|\widehat{u}_{h}-P_ku\|_{0,\O}^{2}+\left[\|\bw-\bw_h\|_{\div,\O}^2+\|u-u_h\|_{0,\O}^2+\|\bw-\bpi_h\bw\|_{0,\O}^2\right]^{2}.
\end{equation*}
Finally, from \eqref{eq:joven-cota}, the above estimate, together with Lemmas~\ref{lm:A1}, \ref{lmm:rodolfo9}, \ref{lmm:rodolfo10}, and Theorem \ref{cotadoblepandeo}, we conclude the proof.
\end{proof}
 
\section{A posteriori error analysis}
\label{sec:apost}
In this section, we develop an a posteriori error
estimator for the acoustic  eigenvalue problem \eqref{def:acoustic_problem}. The a posteriori error estimator that we will propose is of residual type and our goal is to prove that is reliable and efficient. 

 Let us introduce some notations and definitions. For any polygon $\E\in\mathcal{T}_h$
 we denote by $\mathcal{E}_\E$ the set of edges of $\E$ and
\begin{equation*}
\mathcal{E}:=\bigcup_{\E\in\mathcal{T}_h}\mathcal{E}_\E.
\end{equation*}

We decompose $\mathcal{E}$ as $\mathcal{E}:=\mathcal{E}_\O\cup\mathcal{E}_{\Gamma}$, where $\mathcal{E}_{\Gamma}:=\{\ell \in\mathcal{E}\,:\,\ell\subset\Gamma\}$ and $\mathcal{E}_\O:=\mathcal{E}\setminus\mathcal{E}_{\Gamma}$. On the other hand, given $\bxi\in \L^2(\O)^2$, for each $\E\in\CT_h$ and $\ell\in\mathcal{E}_\O$, we denote by $\jump{\bxi\cdot\boldsymbol{t}}$ the tangential jump of $\bxi$ across $\ell$, that is $\jump{\bxi\cdot\boldsymbol{t}}:=(\bxi|_\E-\bxi|_{\E'})|_{\ell}\cdot\boldsymbol{t}$, where $\E$ and $\E'$ are elements of $\CT_h$ having $\ell$ as a common edge.
%
Due the regularity assumptions on the mesh, for each polygon $\E\in\mathcal{T}_h$ there is a sub-triangulation
$\mathcal{T}_h^\E$ obtained by joining each vertex of $\E$ with the midpoint of
the ball with respect to which $\E$ is star-shaped. We define
\begin{equation*}
\displaystyle\widehat{\mathcal{T}}_h:=\bigcup_{\E\in\mathcal{T}_h}\mathcal{T}_h^\E.
\end{equation*}


For each polygon $\E$, we define the following computable and local terms
\begin{equation}
\label{eq:residual_terms}
\boldsymbol{R}_{\E}^{2}:=h_{\E}^2\left\|\rot\bpi_h^\E\bw_h\right\|_{0,\E}^2,\quad
\boldsymbol{\theta}_{\E}^{2}:=S^{\E}(\bw_h-\bpi_h^\E \bw_h,\bw_h-\bpi_h^\E \bw_h),\quad
\boldsymbol{J}_{\ell}:=\jump{\bpi_h^\E\bw_h\cdot\boldsymbol{t}},
\end{equation}
which allows us to define the local error indicator
\begin{equation}
\label{eq:local_ind}
\boldsymbol{\eta}_\E:=\boldsymbol{R}_{\E}^2+\boldsymbol{\theta}_{\E}^{2}+\sum_{\ell\in\mathcal{E}_\E}h_{\E}\left\|\boldsymbol{J}_{\ell}\right\|_{0,\ell}^{2},
\end{equation}
and hence, the global error estimator
\begin{equation}
\label{eq:global_est}
\boldsymbol{\eta}:=\left\{\sum_{\E\in\mathcal{T}_h}\boldsymbol{\eta}_\E^2\right\}^{1/2}.
\end{equation}

In what follows we will prove that \eqref{eq:global_est} is reliable and locally efficient. With this aim,  we begin by decomposing  the error $\bw-\bw_{h}$, using the classical Helmholtz decomposition  as follows:
\begin{equation*}
\bw-\bw_h=\nabla\psi+\curl\beta,
\end{equation*}
with $\psi\in \tilde{\H}^1(\O):=\{v\in\H^1(\O): (v,1)_{0,\O}=0\}$
and $\beta\in \H_{0}^{1}(\O)$. Moreover,  the following regularity result holds
\begin{equation}
\label{eq:controlH}
\|\psi\|_{1,\O}+\|\curl\beta\|_{0,\O}\lesssim \|\bw-\bw_{h}\|_{0,\O}.
\end{equation}
With this decomposition at hand, we split the $\L^2(\O)$-norm
of the error $\bw-\bw_{h}$ into two terms,
\begin{equation}\label{eq:helm}
\Vert\bw-\bw_{h}\Vert_{0,\O}^2=
(\bw-\bw_{h},\nabla\psi)_{0,\O}+
(\bw-\bw_{h},\curl\beta)_{0,\O}.
\end{equation}
To conclude the reliability, we begin by proving the  following results

\begin{lemma}
\label{lmm:grad_bound}
There holds
\begin{equation*}
(\bw-\bw_{h},\nabla\psi)_{0,\O}\lesssim
h^{2\widetilde{r}}\|\bw-\bw_h\|_{0,\O},
\end{equation*}
where $\widetilde{r}\in (\frac{1}{2},1]$ and the hidden constant are independent of $h$.
\end{lemma}
\begin{proof}

An integration by parts reveals that
$$(\bw-\bw_{h},\nabla\psi)_{0,\O}=-(\div(\bw-\bw_{h}),\psi)_{0,\O}+
((\bw-\bw_{h})\cdot\boldsymbol{n},\psi)_{0,\G}.$$
Using that $(\bw-\bw_{h})\cdot\boldsymbol{n}=0$ on $\G$, the fact that $\div\bw=-\lambda u$, $\div\bw_h=-\lambda_h u_h$, and adding and subtracting $\lambda_h(u-P_ku,\psi)_{0,\O}$, we obtain 
\begin{align*}
(\bw-\bw_{h},\nabla\psi)_{0,\O}&=\underbrace{((\lambda-\lambda_h)u,\psi)_{0,\O}}_{\mathbf{I}}
+\underbrace{\lambda_h(u-P_ku,\psi)_{0,\O}}_{\mathbf{II}}+\underbrace{\lambda_h(P_ku-u_h,\psi)_{0,\O}}_{\mathbf{III}}.
\end{align*}

For the term $\mathbf{I}$, we use \eqref{eq:l_lh}, the fact that   $\|u\|_{0,\O}=1$ and \eqref{eq:controlH} in order to obtain
\begin{align*}
(\lambda-\lambda_h)(u,\psi)_{0,\O}&\lesssim h^{2\min\{r,k+1\}}\|\psi\|_{0,\O}
\lesssim h^{2\min\{r,k+1\}}\|\bw-\bw_{h}\|_{0,\O},
\end{align*} 
with $r>\frac{1}{2}$ as in Lemma \ref{lmm:charact}.
Applying the approximation properties \eqref{eq:salim}  and \eqref{eq:controlH} on $\mathbf{II}$,  we obtain
\begin{equation*}
(u-P_ku,\psi)_{0,\O}=(u-P_ku,\psi-P_k\psi)_{0,\O}\lesssim h\|u-P_ku\|_{0,\O}
|\psi|_{1,\O}\lesssim h^2\|\bw-\bw_{h}\|_{0,\O}.
\end{equation*}
Finally, for $\mathbf{III}$, we use Lemma~\ref{lmm:rodolfo11} and the bound for $\|\psi\|_{0,\O}$ to write
\begin{align*}
\lambda_h(P_ku-u_h,\psi)_{0,\O}\lesssim h^{2\widetilde{r}}\|\bw-\bw_{h}\|_{0,\O}.
\end{align*}

Now combining the above estimates we conclude the proof.
\end{proof}
Given $k\in\mathbb{N}\cup\{0\}$,  let us consider  the following virtual discrete subspace of $\H^{1}(\O)$ 
$$V_{h}:=\left\{\zeta\in \H^{1}(\O): \Delta \zeta\in \bbP_{k-1}(\E)\quad  \forall \E\in\CT_{h}, \zeta\in \mathcal{C}(\partial \E): \zeta|_{\ell}\in \bbP_{k+1}(\ell),\quad \forall\text{ edge}\,\, \ell\subset \partial\E\right\}.$$
Then, there exists $\zeta_{I}\in  V_{h}$  that satisfies (see the proof of  \cite[Lemma 3.4]{MRR_apost})
\begin{equation*}
\|\zeta-\zeta_I\|_{0,\ell}\lesssim h_\ell^{1/2}\|\zeta\|_{1,\E}\quad\text{and}\quad\|\zeta-\zeta_I\|_{0,\E}\lesssim h_\E\|\zeta\|_{1,\E}\qquad \forall \zeta\in \H^{1}(\E).
\end{equation*}
With this result at hand, now we prove the following result.
\begin{lemma}
\label{lmm:curl_bound}
There holds
$$(\bw-\bw_{h},\curl\beta)_{0,\O}\lesssim \boldsymbol{\eta}\Vert\curl\beta\Vert_{0,\O},$$
where the hidden constant is independent of $h$ and the discrete solution.
\end{lemma}
\begin{proof}
Since $\curl\beta\in \H_{0}(\div^{0};\O)$,
we have that $(\bw,\curl\beta)_{0,\O}=0$. Thus,
\begin{equation}
\label{eq:int1}
(\bw-\bw_{h},\curl\beta)_{0,\O}
=(\bw_h,\curl\beta)_{0,\O}=(\bw_h-\bpi_h\bw_h,\curl\beta)_{0,\O}
+(\bpi_h\bw_h,\curl\beta)_{0,\O}.
\end{equation}
For the first term term on the right-hand side of the above equality we have
\begin{align}\label{eq:int2}
(\bw_h-\bpi_h\bw_h,\curl\beta)_{0,\O}\lesssim \boldsymbol{\eta}\|\curl\beta\|_{0,\O}.
\end{align}
 
 Next, we introduce $\beta_I$ as the virtual interpolant of $\beta$,
and using that $\curl\beta_h\in \H_{0}(\div^{0};\O)$, we have
$$(\bpi_h^\E\bw_h,\curl\beta)_{0,\O}=(\bpi_h^\E\bw_h,\curl(\beta-\beta_I))_{0,\O}
=\sum_{\E\in\mathcal{T}_h}\int_{\E}\bpi_h^\E\bw_h\cdot\curl(\beta-\beta_I).$$
Now, by using integration by parts, we obtain
$$(\bpi_h^\E\bw_h,\curl\beta)_{0,\O}
=\sum_{\E\in\mathcal{T}_h}\left(\int_{\E}\rot\bpi_h^\E\bw_h(\beta-\beta_I)
+\int_{\partial \E}(\bpi_h^\E\bw_h\cdot\boldsymbol{t})(\beta-\beta_I)\right).$$

Hence, applying Cauchy-Schwarz inequality and property of approximation of $\beta_I$ in the estimate above yields to
\begin{equation}\label{eq:int3}
(\bpi_h^\E\bw_h,\curl\beta)_{0,\O}\lesssim \sum_{\E\in\mathcal{T}_h}\boldsymbol{\eta}_\E\|\beta\|_{1,\omega_\E}\leq C\boldsymbol{\eta}\|\curl\beta\|_{0,\O}.
\end{equation}
Now, combining \eqref{eq:int1}, \eqref{eq:int2} and \eqref{eq:int3}  we conclude the proof.

\end{proof}
We now provide an upper bound for our error estimator.
\begin{lemma}
\label{reliability}
The following error estimate holds
\begin{align*}\label{eq:1estimate}
\|\bw-\bw_{h}\|_{0,\O}&\lesssim \boldsymbol{\eta}+h^{2\widetilde{r}},
\end{align*}
where  the hidden constants are independent of $h$ and the discrete solution.
\end{lemma}
\begin{proof}
The proof is a consequence of \eqref{eq:helm}, Lemmas \ref{lmm:grad_bound} and \ref{lmm:curl_bound}, together to \eqref{eq:controlH}. 
\end{proof}


%
Thanks to the previous lemmas, we have the following result
\begin{lemma}
\label{lmm:cotpi}
The following error estimate holds
\begin{equation*}
\|\bw-\bpi_h\bw_{h}\|_{0,\O}\lesssim \boldsymbol{\eta}+h^{2\widetilde{r}},
\end{equation*}
where  the hidden constant is independent of $h$.
\end{lemma}
\begin{proof}
From the triangle inequality, together to \eqref{eq:20},   for the stability of the $\bpi_{h}$-projector and Lemma \ref{reliability}, we have
\begin{align*}
\|\bw-\bpi_h\bw_{h}\|_{0,\O}&\leq \|\bw-\bw_{h}\|_{0,\O}+\|\bw_{h}-\bpi_h\bw_{h}\|_{0,\O}\\
&\lesssim  \boldsymbol{\eta}+h^{2\widetilde{r}}.
\end{align*}
Hence, we conclude the proof.
\end{proof}

Now we are in position to establish the reliability of our estimator.
\begin{corollary}\label{redeal}[Reliability]
The following error estimate hold
\begin{equation*}
\label{eq:cor41}
\|\bw-\bw_{h}\|_{0,\O}+\|\bw-\bpi_h\bw_{h}\|_{0,\O}\lesssim \boldsymbol{\eta}+h^{2\widetilde{r}}.
\end{equation*}
where  the hidden constants are independent of $h$.
\end{corollary}
\begin{remark}
\label{remrk:HOT}
From Corollary~\ref{redeal}, we note that  $\mathcal{O}(h^{2\widetilde{r}})$
can be considered a ``higher order term" when lowest order VEM ($k = 0$)  is  used.
When $k\ge1$, the term can be considered a ``higher order term"
when the eigenfunction is singular. This usualy happens when the eigenproblem
is solved in non-convex polygonal domains.
\end{remark}
\subsection{Efficiency}

Now our aim is to prove that the local  indicator $\boldsymbol{\eta}_\E$ defined in \eqref{eq:local_ind} provides a lower bound of the error  $\bw-\bw_h$ in a vicinity of any polygon $\E$. To do this task, we procede as is customary for the efficiency analysis, using suitable bubble functions for the polygons and their edges.

The bubble functions that we will consider are based in \cite{CGPS}. Let $\psi_\E\in \H_0^1(\O)$ be an interior bubble function defined in a polygon $\E$. These bubble functions can be constructed piecewise as the sum of the cubic bubble functions for each triangle of the sub-triangulation $\mathcal{T}_h^\E$ that attain the value $1$ at the barycenter of each triangle. Also, the edge bubble function $\psi_\ell\in\partial \E$ is a piecewise quadratic function attaining the value of $1$ at the barycenter of $\ell$ and vanishing on the triangles $\E\in\widehat{\mathcal{T}}_h$ that do not contain $\ell$ on its boundary.

The following technical results for the bubble functions are a key point to prove the efficiency bound.

\begin{lemma}
\label{lmm:bubble1}
For any $\E\in\mathcal{T}_h$, let $\psi_\E$ be the corresponding interior bubble function. Then, there hold
\begin{align*}
\displaystyle  \|p\|_{0,\E}^2&\lesssim\int_\E\psi_\E p^2\lesssim \|p\|_{0,\E}^2\qquad\forall p\in\mathbb{P}_k(\E);\\
\displaystyle  \|p\|_{0,\E}&\lesssim\|\psi_\E p\|_{0,\E}+h_\E\|\nabla(\psi_\E p)\|_{0,\E}\lesssim\|p\|_{0,\E}\qquad\forall p\in\mathbb{P}_k(\E);
\end{align*}
where the hidden constants are independent of $h_{\E}$
\end{lemma}

\begin{lemma}
\label{lmm:bubble2}
For any $\E\in\mathcal{T}_h$ and $\ell\in\mathcal{E}_\E$, let $\psi_\ell$ be the corresponding edge bubble function.  Then, there holds
\begin{equation*}
\displaystyle \|p\|_{0,\ell}^2\lesssim\int_{\ell}\psi_\ell p^2\lesssim \|p\|_{0,\ell}^2\qquad\forall p\in\mathbb{P}_k(\ell).
\end{equation*}
Moreover, for all $p\in\mathbb{P}_k(\ell)$, there exists an extension of $p\in\mathbb{P}_k(\E)$, which we denote simply by $p$, such that
\begin{equation*}
h_\E^{-1/2}\|\psi_\ell p\|_{0,\E}+h_\E^{1/2}\|\nabla (\psi_\ell p)\|_{0,\E}\lesssim \|p\|_{0,\ell},
\end{equation*}
where the hidden constants are independent of $h_{\E}$.
\end{lemma}

Now we are in position to establish  the main result of this section.
\begin{theorem}
\label{thm:efficiency}
For any $\E\in\mathcal{T}_h$, there holds
\begin{equation*}
\boldsymbol{\eta}_\E\lesssim \|\bw_h-\bw\|_{0,\omega_\ell}+\|\bw-\bpi_h^\E\bw_h\|_{0,\omega_\ell}, 
\end{equation*}
where $\omega_\ell$ denotes the union of two polygons sharing an
edge with $\E$, and the hidden constant is independent of $h$ and the discrete solution.
\end{theorem}
\begin{proof}
The aim is to estimate each term of the local indicator \eqref{eq:local_ind}. The proof is divided in three steps:
\begin{itemize}
\item \textbf{Step 1:} We begin by estimating $\boldsymbol{R}_\E^2$ in \eqref{eq:residual_terms}. Invoking  the properties of the bubble function $\psi_\E$, Cauchy-Schwarz inequality,  and Lemma \ref{lmm:bubble1}, we have
\begin{multline*}
\displaystyle \boldsymbol{R}_\E^2\lesssim\int_{\E}\psi_\E\rot(\bpi_h^{\E}\bw_{h})\rot(\bpi_h^{\E}\bw_h)=\int_{\E}\psi_\E\rot(\bpi_h^{\E}\bw_{h})\rot(\bpi_h^{\E}\bw_h-\bw_{h})\\
=-\int_\E(\bpi_h^\E\bw_h-\bw_h)\curl(\psi_\E\rot(\bpi_h^\E\bw_h))
\lesssim\|\bpi_h^\E\bw_h-\bw_h\|_{0,\E}h_\E^{-1} \left\|\rot(\bpi_h\bw_h)\right\|_{0,\E},
\end{multline*}
which implies that
\begin{equation}
\label{eq:contribution1}
\boldsymbol{R}_{\E}=h_\E\|\rot(\bpi_h^{\E}\bw_h)\|_{0,\E}\lesssim\|\bw-\bw_h\|_{0,\E}+\|\bw-\bpi_h^{\E}\bw_h\|_{0,\E}.
\end{equation}
\item\textbf{Step 2:}  Now we estimate  $\boldsymbol{J}_h$. Following the proof  of \cite[Lemma 5.16]{CM2019}, we obtain
\begin{equation*}
\|\boldsymbol{J}_\ell\|_{0,\ell}^2\lesssim \|\psi_\ell \boldsymbol{J}_\ell\|_{0,\ell}^2=\int_{\ell}(\psi_\ell \boldsymbol{J}_\ell)\cdot \boldsymbol{J}_\ell=\int_{\omega_\ell}(\bw-\bpi_h^{\E}\bw_{h})\cdot\curl(\psi_\ell \boldsymbol{J}_\ell)
+\int_{\omega_\ell} \psi_\ell \boldsymbol{J}_\ell\rot\bpi_h^{\E}\bw_h.
\end{equation*}

Hence, from Cauchy-Schwarz inequality, the bubble function properties  and  \eqref{eq:contribution1}, we have 
\begin{align*}
\|\boldsymbol{J}_\ell\|_{0,\ell}^2&\lesssim|\psi_\ell \boldsymbol{J}_\ell|_{1,\omega_\ell}\|\bw-\bpi_h^\E\bw_h\|_{0,\omega_\ell}+\|\psi_\ell \boldsymbol{J}_\ell\|_{0,\omega_\ell}\|\rot\bpi_h^\E\bw_h\|_{0,\omega_\ell},\\
&\lesssim h_\E^{-1/2}\left(\|\bw_h-\bw\|_{0,\omega_\ell}+\|\bw-\bpi_h^\E\bw_h\|_{0,\omega_\ell}\right)\|\boldsymbol{J}_\ell\|_{0,\ell}.
\end{align*}

Hence, we conclude that
\begin{equation}
\label{eq:jump}
h_\E^{1/2}\|\boldsymbol{J}_\ell\|_{0,\ell}\lesssim \|\bw_h-\bw\|_{0,\omega_\ell}+\|\bw-\bpi_h^\E\bw_h\|_{0,\omega_\ell}.
\end{equation}

\item\textbf{Step 3:} The final step is to control the term $\boldsymbol{\theta}_\E$. To do this task,
we use the stability property \eqref{eq:20},  add and subtract $\bw$
with the purpose of applying triangular inequality as follows
\begin{equation}
\label{eq:contribution3}
\boldsymbol{\theta}_{\E}\leq c_{1}\|\bw_h-\bpi_h^\E \bw_h\|_{0,\E}\lesssim
\|\bw_h-\bw\|_{0,\E}+\|\bpi_h^\E \bw_h-\bw\|_{0,\E}.
\end{equation}
\end{itemize}

Hence, the proof is complete by gathering \eqref{eq:contribution1}, \eqref{eq:jump} and \eqref{eq:contribution3}.
\end{proof}

As a direct consequence of lemma above, we have the following result that allows us to
conclude the efficiency of the local and global error estimators
for the acoustic problem, and hence, for its equivalent mixed  problem.
\begin{corollary}\label{Effici}[Efficiency]
 There holds
\begin{equation*}
\boldsymbol{\eta}\lesssim  \|\bw-\bw_{h}\|_{0,\O}+\|\bw-\bpi_h\bw_{h}\|_{0,\O},
 \end{equation*}
 where  the hidden constants are independent of $h$.
\end{corollary}

\section{Numerical results}
\label{sec:numerics}
In this section, we  report numerical tests in order to assess the behavior of the a posteriori estimator 
defined in \eqref{eq:global_est}. With this aim, we have implemented in a MATLAB code a lowest order VEM scheme
on arbitrary polygonal meshes. 

%
 
We have used  the mesh refinement algorithm described in \cite{BM2015},  which consists in splitting each element of the mesh into $n$ quadrilaterals ($n$ being the number of edges of the polygon) by connecting the barycenter of the element with the midpoint of each edge, which will be named as  \textbf{Adaptive VEM}. Notice that although this process is initiated with a mesh of triangles, the successively created meshes will contain other kind of convex polygons, as it can be seen in Figure \ref{FIG:VEM}. Both schemes are based on the strategy of refining those elements $\E\in\CT_h$ that satisfy
 $$\boldsymbol{\eta}_{\E}\geq 0.5 \max_{{\E'\in\CT_{h}}}\{\boldsymbol{\eta}_{\E'}\}.$$

\subsection{Test 1: L-shaped domain.}
We will consider the non-convex domain 
$\O:=\left(0,1\right )\times \left(0,1\right)\setminus [1/2,1 ] \times [1/2,1 ]$.

It is clear that the first eigenfunction of the acoustic problem
in this domain is not smooth enough, due the presence of a geometrical
singularity at $\left(\frac{1}{2}, \frac{1}{2}\right)$. This leads to a lack
of regularity due to the 
reentrant angle $\omega=\pi/3$. Therefore, according to \cite{BeiraoVEMAcoustic2017},
using quasi-uniform meshes, the convergence rate for the eigenvalues should be $|\l-\l_{h}|=\mathcal{O}(h^{4/3})\approx\mathcal{O}(N^{-2/3})$, where $N$ denotes the number of degrees of freedom. Then, the proposed a posteriori estimator \eqref{eq:global_est} must be capable to recover the optimal order $|\l-\l_{h}| = \mathcal{O} (N^{-1})$, when the adaptive refinement is performed near to the singularity point. 

For the numerical tests, we have computed the smallest eigenvalue and its  corresponding eigenfunction using the MATLAB command \texttt{eigs}.

Figures \ref{FIG:VEM} to \ref{FIG:VEM2} show the adaptively refined meshes obtained with VEM procedures and different initial meshes. Figure \ref{FIG:VEM} is initiated with a mesh of triangles, while Figure \ref{FIG:VEM2} is initiated with a  non-structured hexagonal meshes made of convex hexagons.
\begin{figure}[H]
\begin{center}
\begin{minipage}{4.2cm}
\centering\includegraphics[height=4.1cm, width=4.1cm]{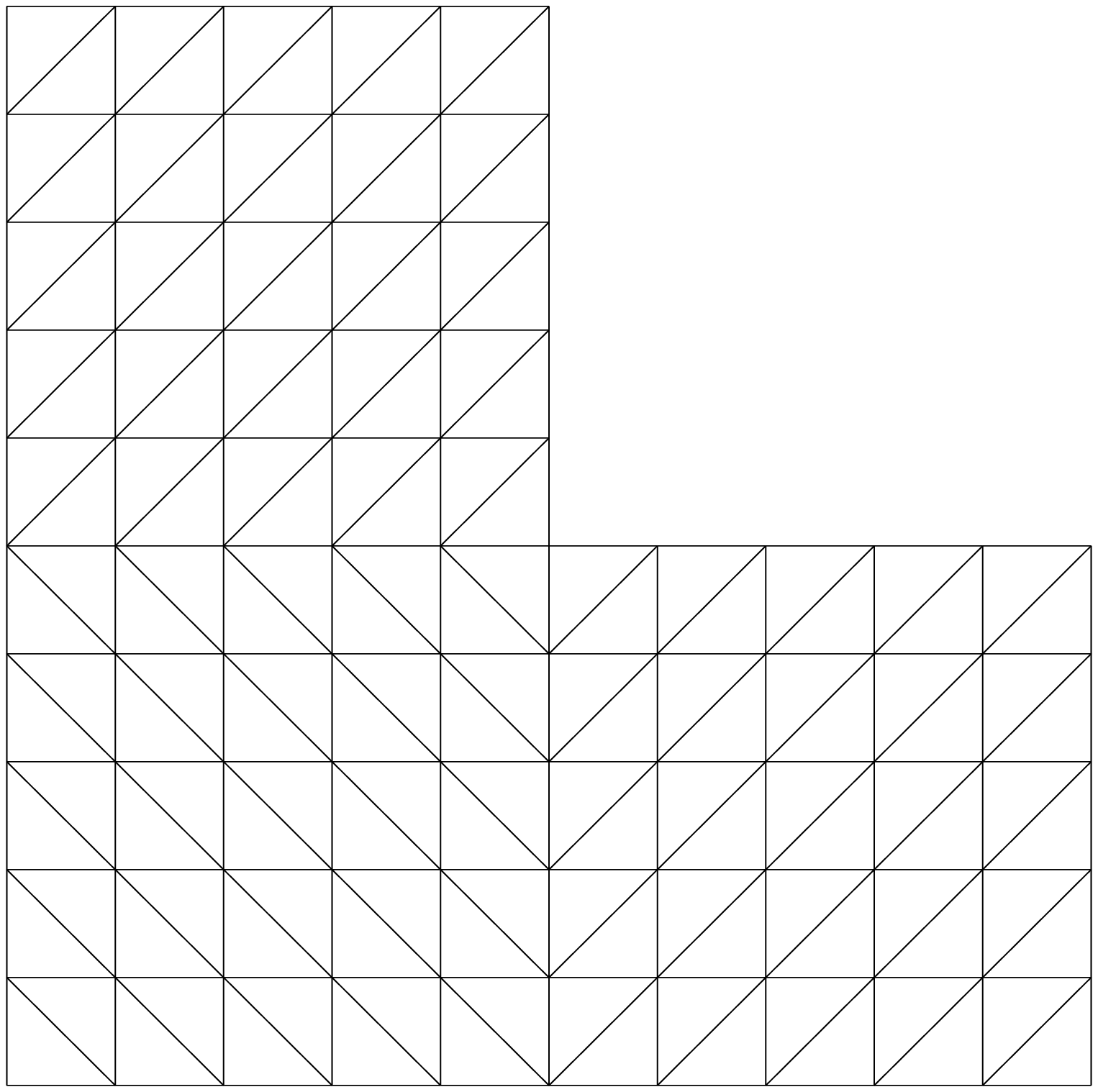}
\end{minipage}
\begin{minipage}{4.0cm}
\centering\includegraphics[height=4.1cm, width=4.1cm]{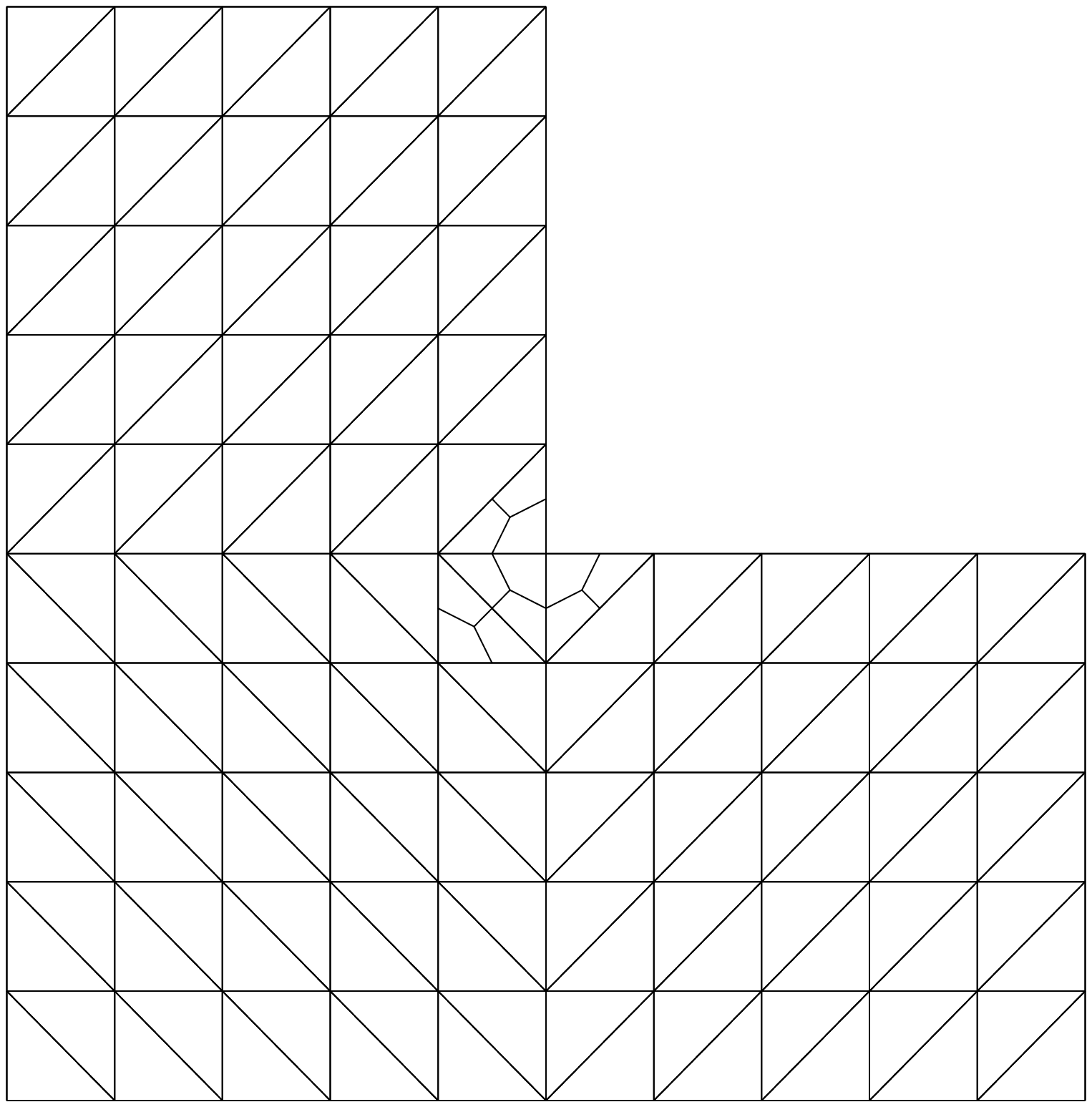}
\end{minipage}
\begin{minipage}{4.0cm}
\centering\includegraphics[height=4.1cm, width=4.1cm]{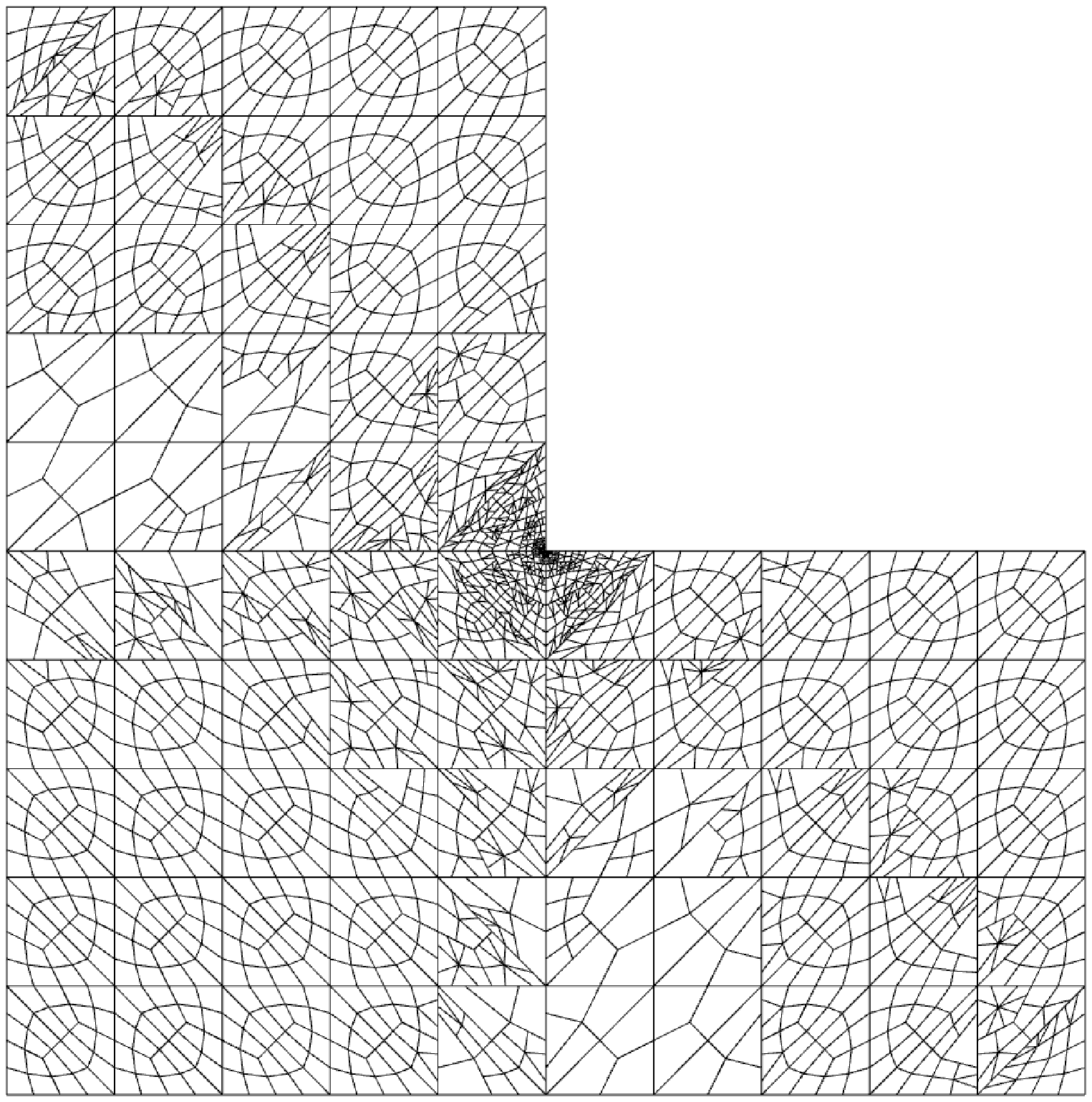}
\end{minipage}
\caption{Test 1. Adaptively refined meshes obtained with VEM scheme at refinement steps 0, 1 and 8 (Adaptive VEM).}
\label{FIG:VEM}
\end{center}
\end{figure}
\begin{figure}[H]
\begin{center}
\begin{minipage}{4.2cm}
\centering\includegraphics[height=4.1cm, width=4.1cm]{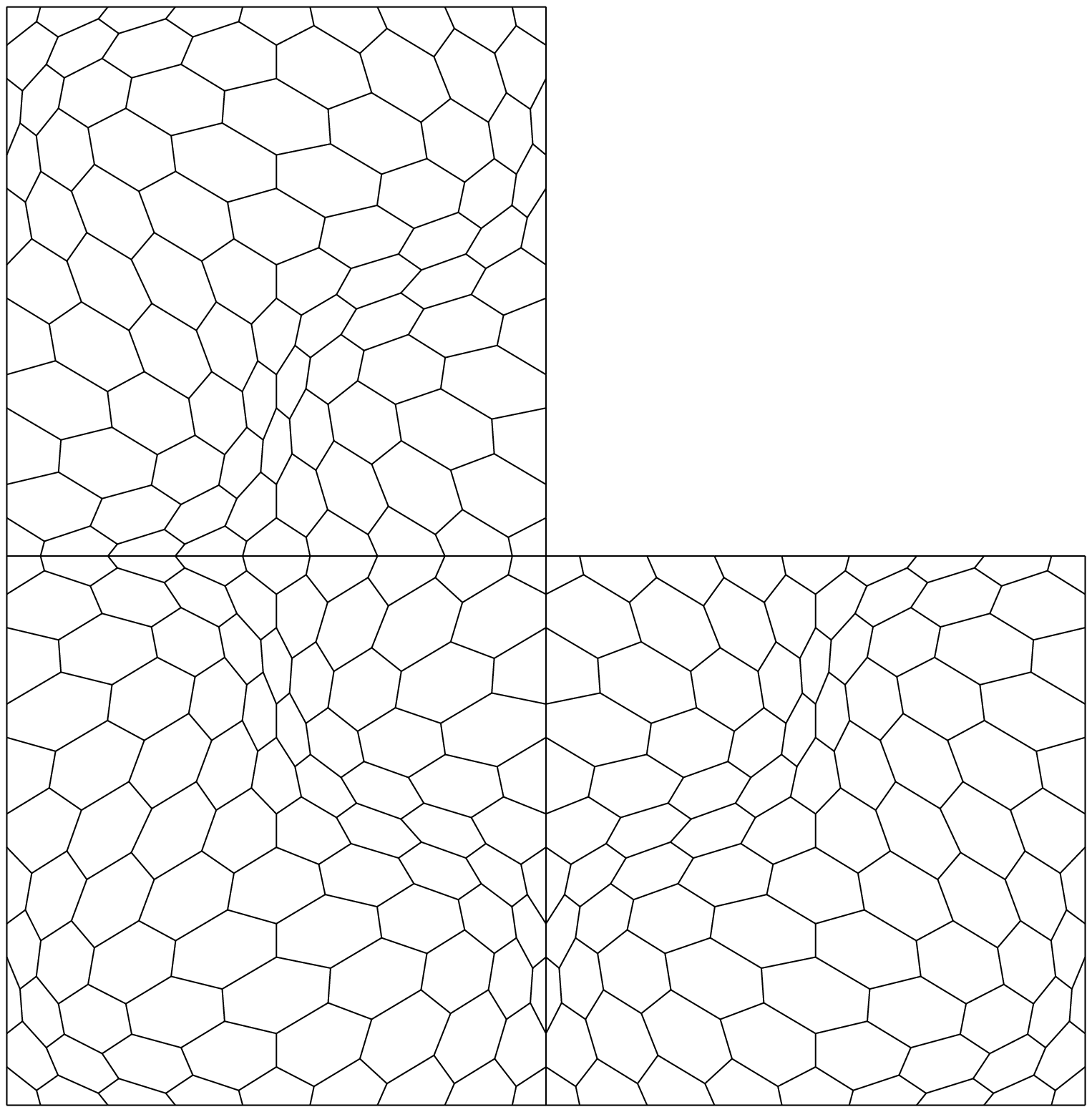}
\end{minipage}
\begin{minipage}{4.2cm}
\centering\includegraphics[height=4.1cm, width=4.1cm]{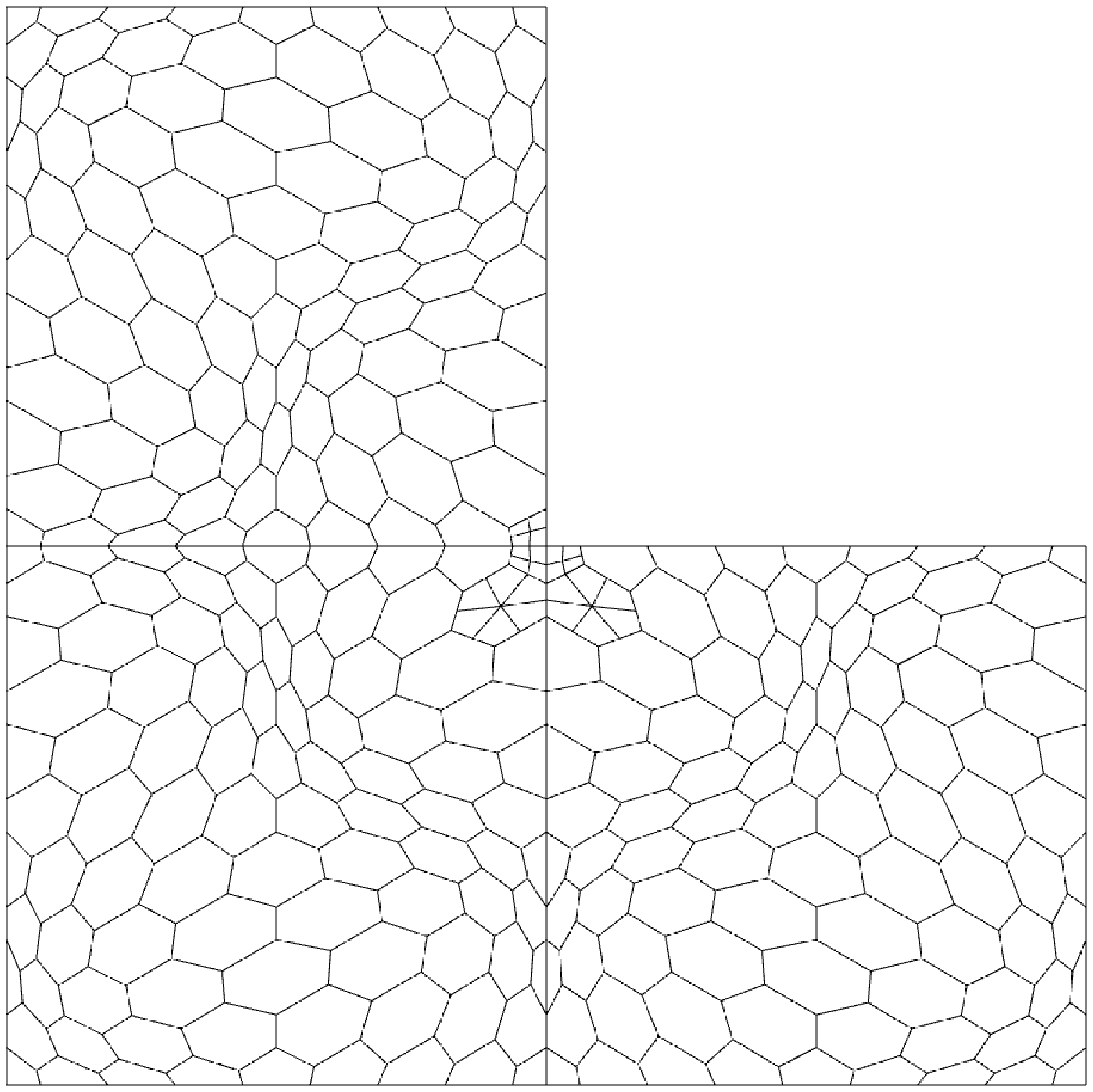}
\end{minipage}
\begin{minipage}{4.2cm}
\centering\includegraphics[height=4.1cm, width=4.1cm]{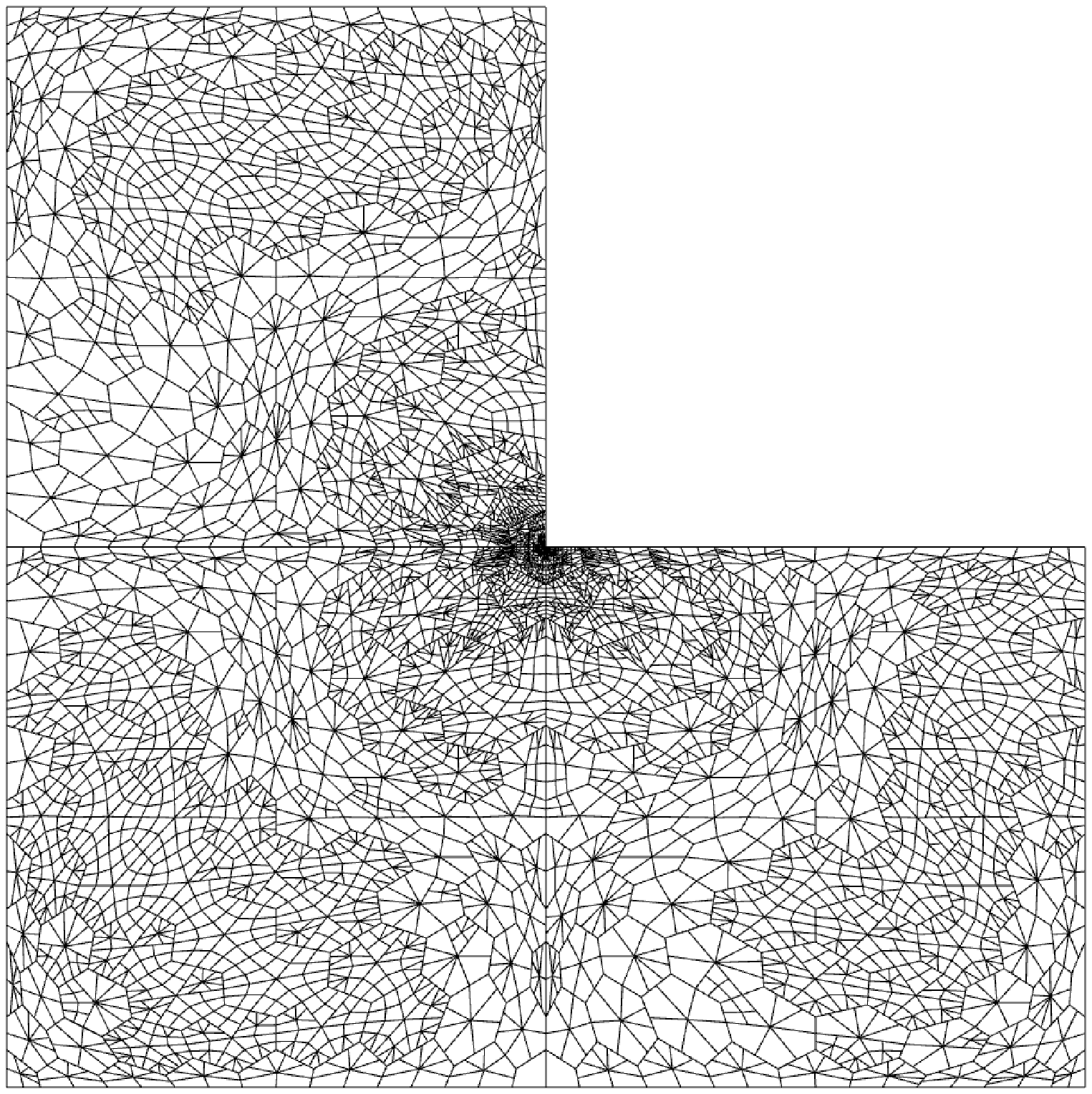}
\end{minipage}
\caption{Test 1. Adaptively refined meshes obtained with VEM scheme at refinement steps 0, 1 and 8 (Adaptive VEM).}
\label{FIG:VEM2}
\end{center}
\end{figure}

Figures \ref{FIG:VEM} and \ref{FIG:VEM2} show that our estimator identifies the singularity point of the domain, leading to a refinement 
on the region of the re-entrant angle. This refinement allows to achieve the optimal order of convergence for the eigenvalue.

In order to compute the errors $|\l_{1}-\l_{h1}|$, and since an exact eigenvalue is not known, we have used an approximation based on a least-squares fitting of the computed values obtained with extremely refined meshes. Thus, we have obtained the value $\l_{1}=5.9017$, which has at least four accurate significant digits.

We report in Table \ref{TABLA:1} the lowest eigenvalue $\l_{h1}$ on uniformly refined meshes, adaptively refined meshes with VEM schemes and in the last column we report adaptively refined meshes with VEM schemes and initial non-structured hexagonal meshes. Each table includes the estimated convergence rate.

\begin{table}[H]
\begin{center}
\caption{Test 1. Computed lowest eigenvalue   $\l_{h1}$ computed with different schemes.}
\begin{tabular}{|c|c||c|c||c|c||c|c|c|c}
  \hline
    \multicolumn{2}{|c||}{Uniform VEM} & \multicolumn{2}{c||}{Adaptive VEM} &  \multicolumn{2}{c|}{Adaptive VEMH}  \\
    \hline
     $N$ & $\l_{h1}$  &   $N$ & $\l_{h1}$  &       $N$ & $\l_{h1}$\\
\hline
 245 &  5.6831  &   245   &5.6831  & 829  & 5.8283\\
   940 &  5.8231 & 266  & 5.7356 &  872  & 5.8495\\
   3680  & 5.8732&  288 &  5.7554  & 945  & 5.8605\\
   14560  & 5.8914 &   381 &  5.7805 &  1131  & 5.8688\\
   57920  & 5.8982 &   889 &  5.8440 &  2010  & 5.8833\\
   231040  & 5.9008 &  1206 &  5.8620 &  3296  & 5.8895\\
                     &    &   1731 &  5.8713  & 4932&   5.8928\\
                     &    &   3639 &  5.8876 &  7287&   5.8955\\
                     &     &  5206 &  5.8924&   12003&   5.8982\\
                     &    &    7653&   5.8949&   19349&  5.8998\\
                      &   &    14545&   5.8985&   30751 &  5.9007\\
                     &     &  22982 &  5.9000&   50421 &  5.9014\\
                     &     &  33844&   5.9006  &                 &     \\
                      &    &  61641&   5.9015  &                 &     \\
     \hline 
     Order   &$\mathcal{O}\left(N^{-0.79}\right)$&       Order   & $\mathcal{O}\left(N^{-1.08}\right)$ &   Order   & $\mathcal{O}\left(N^{-1.14}\right)$\\
       \hline
       $\l_1$  &5.9017 &  $\l_{1}$  &5.9017&   $\l_1$  &5.9017\\
     \hline
    \end{tabular}
\label{TABLA:1}
\end{center}
\end{table}    
\begin{figure}[ht]
\begin{center}
\begin{minipage}{5.0cm}
\centering\includegraphics[height=8.0cm, width=8.0cm]{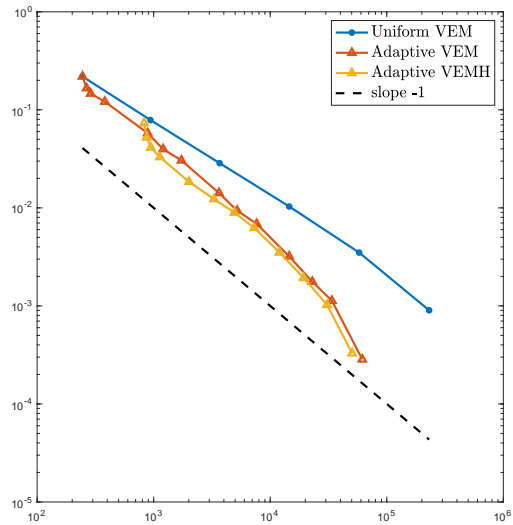}
\end{minipage}
\caption{Test 1. Error curves of $|\l_{1}-\l_{h1}|$ for uniformly refined meshes
(``Uniform VEM''), adaptively refined meshes with VEM (``Adaptive VEM'') and adaptively refined meshes with VEM and  initial mesh of hexagons (``Adaptive VEMH'').}
\label{ero234}
\end{center}
\end{figure}

In Figure \ref{ero234} we present error curves where we observe that the two refinement schemes lead to a correct convergence rate.
It can be seen from Table \ref{TABLA:1} and Figure \ref{ero234}, that the uniform refinement leads to a convergence rate close to that predicted by the theory, while the adaptive VEM  schemes allow us to recover the optimal order of convergence $\mathcal{O}\left(N^{-1}\right)$.

We report in Table \ref{TABLA:2}, the error $|\l_{1}-\l_{h1}|$ and the estimators $\boldsymbol{\eta}$ at each step of the adaptive
VEM scheme. We include in the table the  terms  $\displaystyle \boldsymbol{\theta}^{2}:=\sum_{\E\in\CT_{h}}\boldsymbol{\theta}_{\E}^{2},$
which appears from the inconsistency of the VEM, and 
$\boldsymbol{J}_h:=\displaystyle\sum_{\E\in\CT_{h}}\left(\sum_{\ell\in\mathcal{E}_\E} h_{\E}\|\boldsymbol{J}_{\ell}\|_{0,\ell}^{2}\right),$
which arise from the edge residuals.  We also report in the table the effectivity indexes $|\l_{1}-\l_{h1}|/\boldsymbol{\eta}^{2}$.

\begin{table}[H]
\begin{center}
\caption{Components of the error estimator and effectivity indexes on the adaptively refined meshes with VEM.}
\vspace{0.3cm}
\begin{tabular}{|c|c|c|c|c|c|c|c|c|}
\hline
$N$   & $\l_{h1}$ &  $|\l_{1}-\l_{h1}|$   & $\boldsymbol{\theta}^{2}$ & $\boldsymbol{J}_h^{2}$ &  $\boldsymbol{\eta}^2$ & $\dfrac{|\l_{1}-\l_{h1}|}{\boldsymbol{\eta}^2}$ \\
\hline
 245   &5.6831  & 2.1869e-01  & 9.4718e-03   &1.3226e-01  & 1.4174e-01 &  1.5429\\
   266&   5.7356 & 1.6614e-01  & 1.2881e-02  & 8.5612e-02 &  9.8493e-02 &  1.6868\\
   288  & 5.7554 &  1.4637e-01  & 1.2791e-02 &  7.5101e-02 &  8.7892e-02  & 1.6653\\
   381  & 5.7805 &  1.2119e-01  & 1.2944e-02 &  5.6072e-02 &  6.9016e-02  & 1.7560\\
   889  & 5.8440 &  5.7776e-02 &  9.4599e-03 &  1.5576e-02  & 2.5036e-02  & 2.3077\\
   1206  & 5.8620&  3.9753e-02  & 6.8149e-03 &  1.0962e-02  & 1.7777e-02  & 2.2362\\
   1731  & 5.8713 &  3.0451e-02   &5.1742e-03 &  8.1822e-03  & 1.3356e-02  & 2.2799\\
   3639  & 5.8876 &  1.4166e-02  & 2.7744e-03 &  3.2286e-03  & 6.0030e-03  & 2.3598\\
   5206 &  5.8924 &  9.3787e-03  & 1.8562e-03 &  2.3210e-03  & 4.1771e-03  & 2.2453\\
   7653 &  5.8949&  6.8767e-03  & 1.3746e-03 &  1.6477e-03  & 3.0223e-03  & 2.2753\\
   14545 &  5.8985  & 3.1983e-03  & 7.4205e-04 &  9.1237e-04  & 1.6544e-03  & 1.9332\\
   22982 &  5.9000 &  1.7633e-03  & 4.6344e-04 &  5.9698e-04  & 1.0604e-03  & 1.6628\\
   33844 &  5.9006  & 1.1283e-03  & 3.3863e-04  & 4.2666e-04  & 7.6529e-04  & 1.4743\\\hline
\end{tabular}
\label{TABLA:2}
\end{center}
\end{table}

From Table \ref{TABLA:2} we observe  that the effectivity indexes are bounded and far from zero. Also, the inconsistency and edge residual terms are, roughly speaking, of the same order. This results are similar to those obtained
in \cite{MRR_apost}. We end this test presenting in Figure~\ref{FIG:eigenfunction}  the displacement field and the pressure fluctuation of the fluid on the L-shaped domain, associated to the first eigenfunction.
\begin{figure}[H]
\begin{center}
\begin{minipage}{4.2cm}
\centering\includegraphics[height=4.1cm, width=4.1cm]{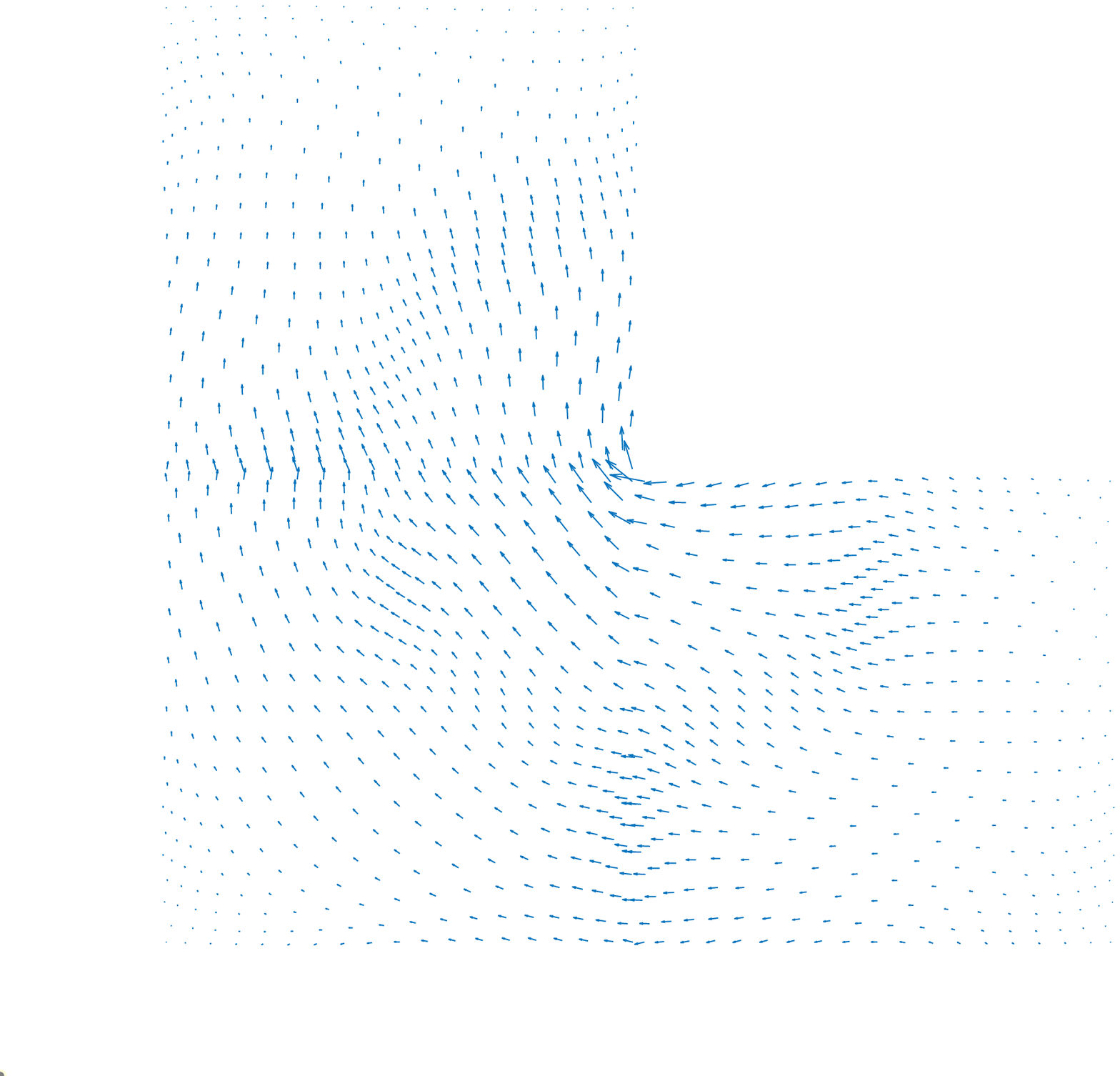}
\end{minipage}
\begin{minipage}{4.2cm}
\centering\includegraphics[height=4.1cm, width=4.1cm]{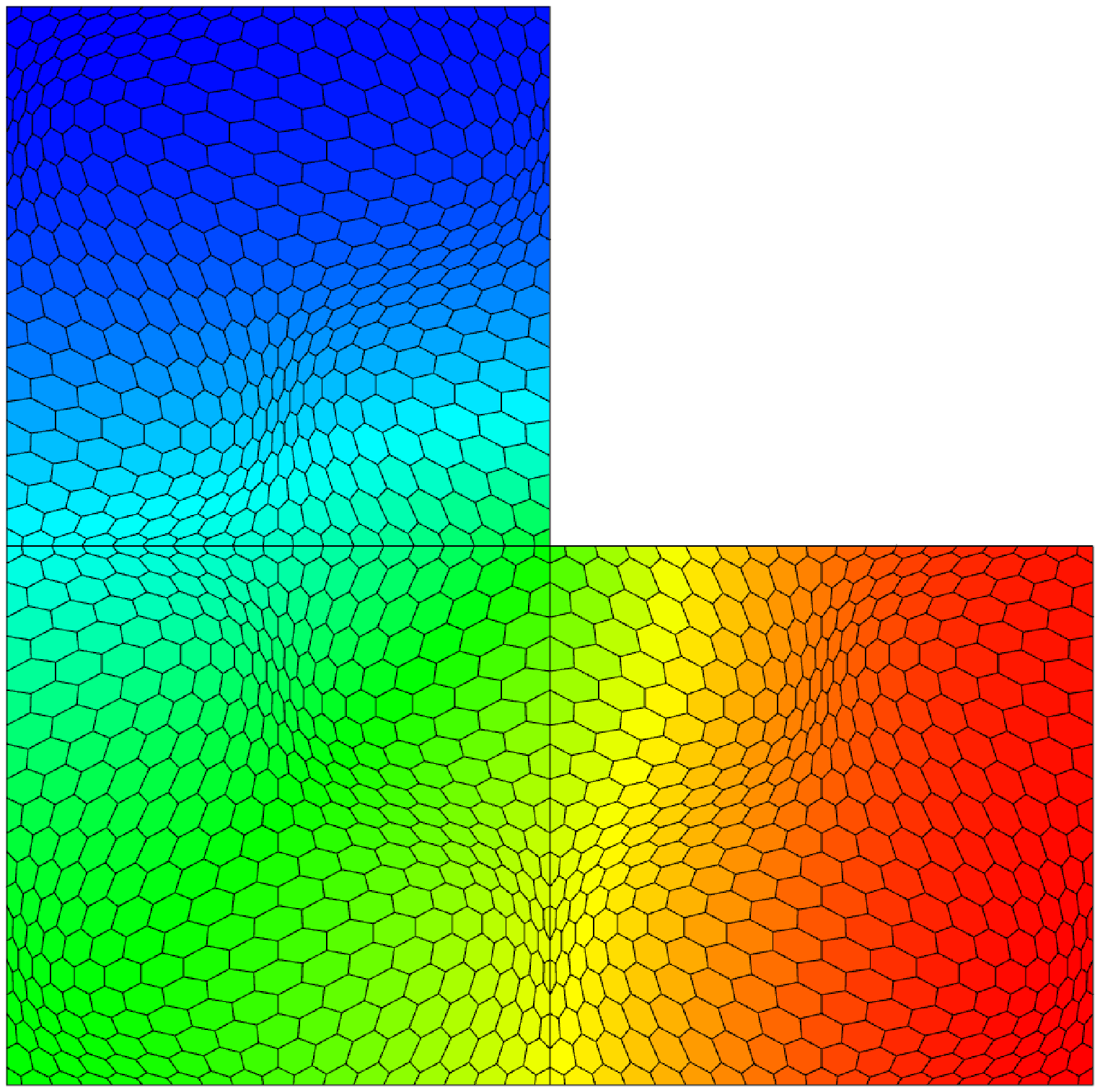}
\end{minipage}
\caption{Test 1. Eigenfunctions of the acoustic problem corresponding to the first lowest eigenvalue: displacement field $\bw_{h1}$ (left), pressure fluctuation $p_{h1}$ (right).}
\label{FIG:eigenfunction}
\end{center}
\end{figure}

\subsection{Test 2: H-shaped domain.}
The aim of this test is to assess the performance of the adaptive scheme when solving a problem with a singular solution.
In this test $\Omega$ consists of an H-shaped domain that represents the union of two pools. 
More precisely, the geometry of this domain is given by
\begin{equation*}
\Omega:=\big\{ (0,3/2)\times(0,3) \big\} \setminus \big\{ \{ [1/2,1]\times [0,5/4]  \} \cup \{ [1/2,1]\times [15/8,3] \}  \big\}.
\end{equation*}
According to the definition of this domain, four singularities are present, leading once again to a lack of regularity for the eigenfunctions
of our acoustic problem. Hence, the proposed estimator  $\boldsymbol{\eta}$ defined in \eqref{eq:global_est} must be capable of 
identify these singularities of the geometry and perform and adaptive refinement, with different polygonal meshes, in order to recover optimal order of convergence.

Figures \ref{FIG:VEMHST} to \ref{FIG:VEMHV} show the adaptively refined meshes obtained with VEM procedures and different initial meshes. In Figure \ref{FIG:VEMHST} we start with a mesh of triangles and squares, while in Figure \ref{FIG:VEMHB} we begin with a trapezoidal mesh consisting of partitions of the domain into $M\times M$ congruent trapezoids. Finally in Figure \ref{FIG:VEMHV} we start with a Voronoi mesh.
\begin{figure}[H]
\begin{center}
\begin{minipage}{4.2cm}
\centering\includegraphics[height=6.1cm, width=4.1cm]{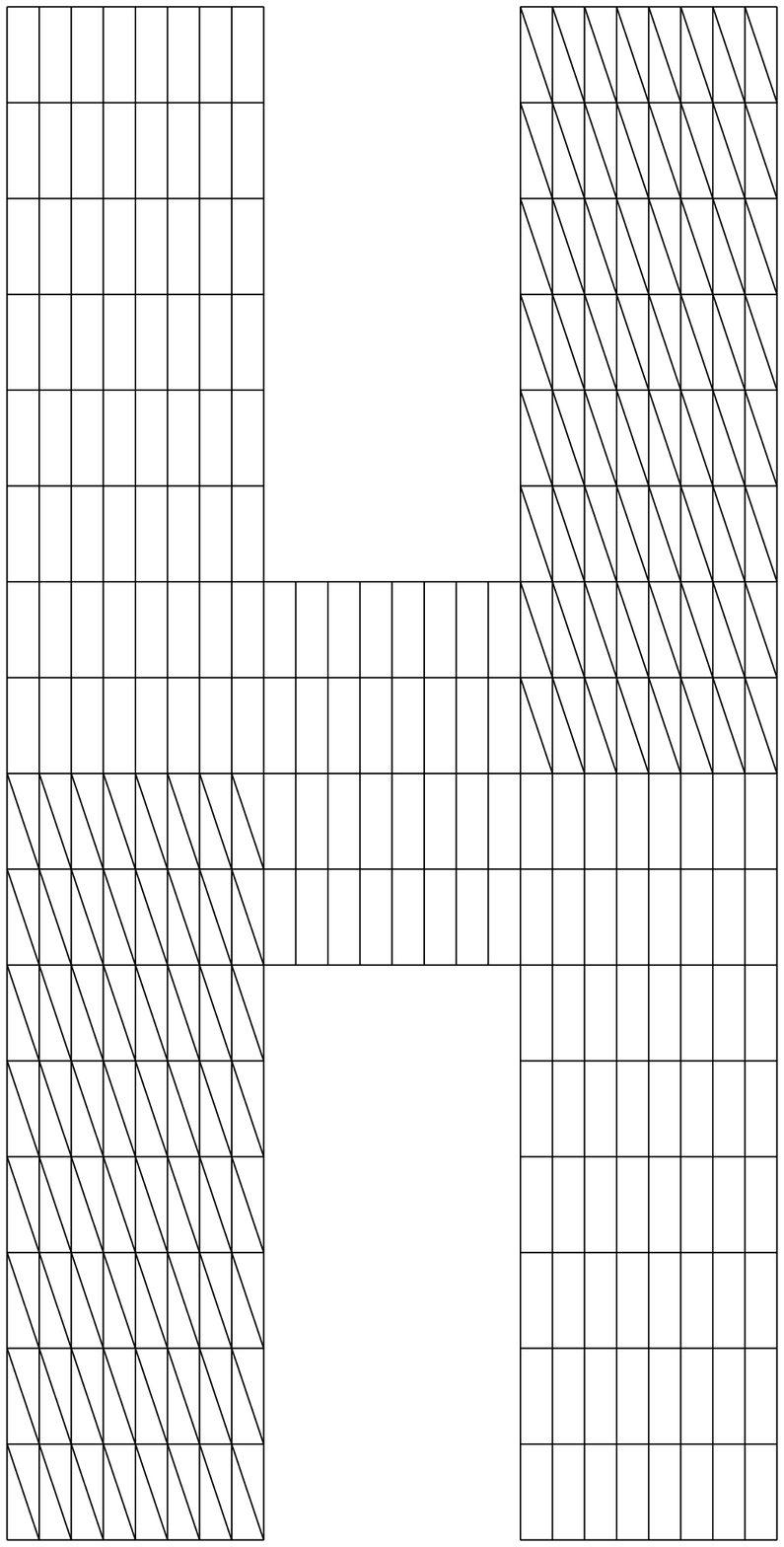}
\end{minipage}
\begin{minipage}{4.2cm}
\centering\includegraphics[height=6.1cm, width=4.1cm]{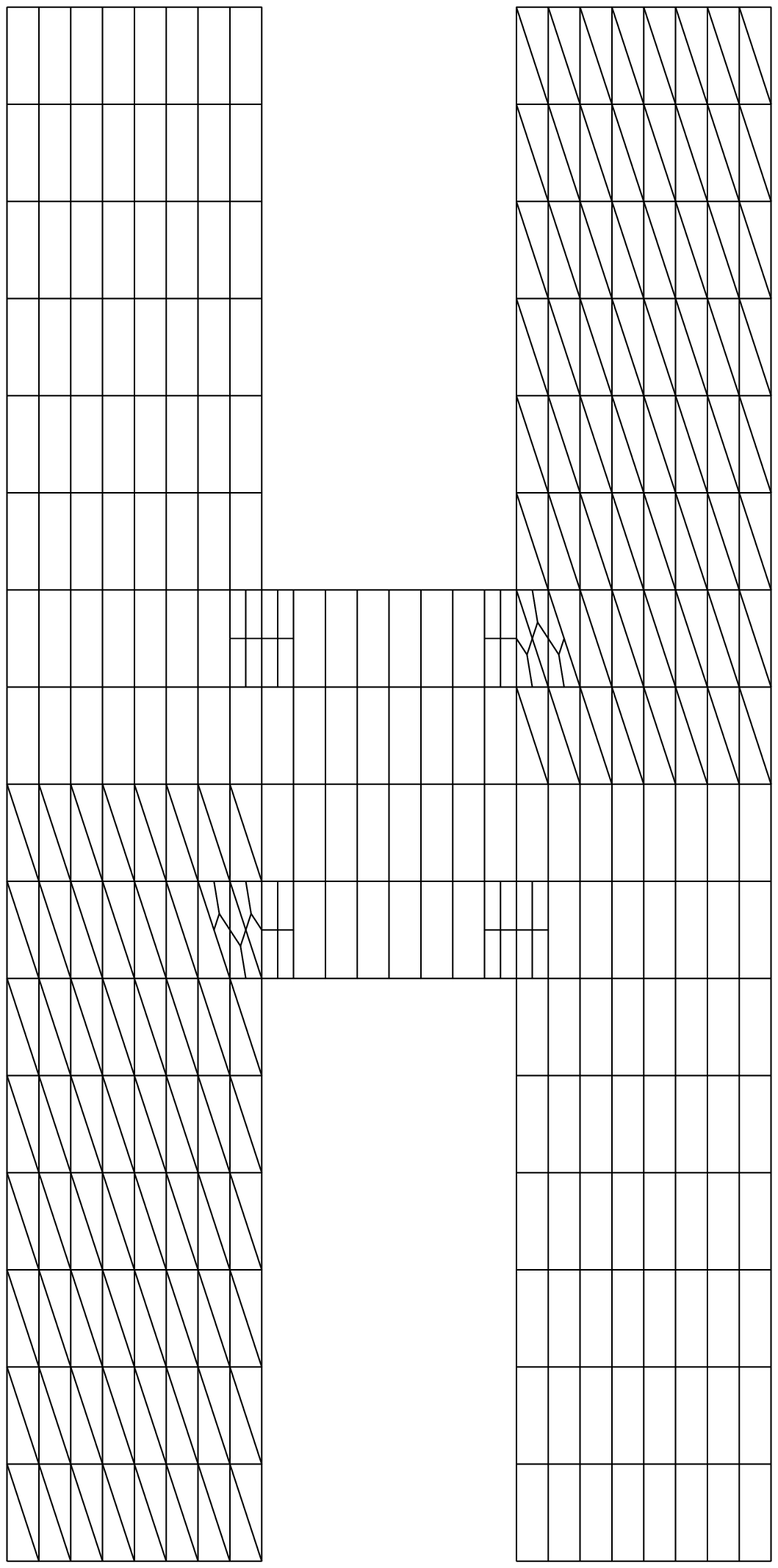}
\end{minipage}
\begin{minipage}{4.2cm}
\centering\includegraphics[height=6.1cm, width=4.1cm]{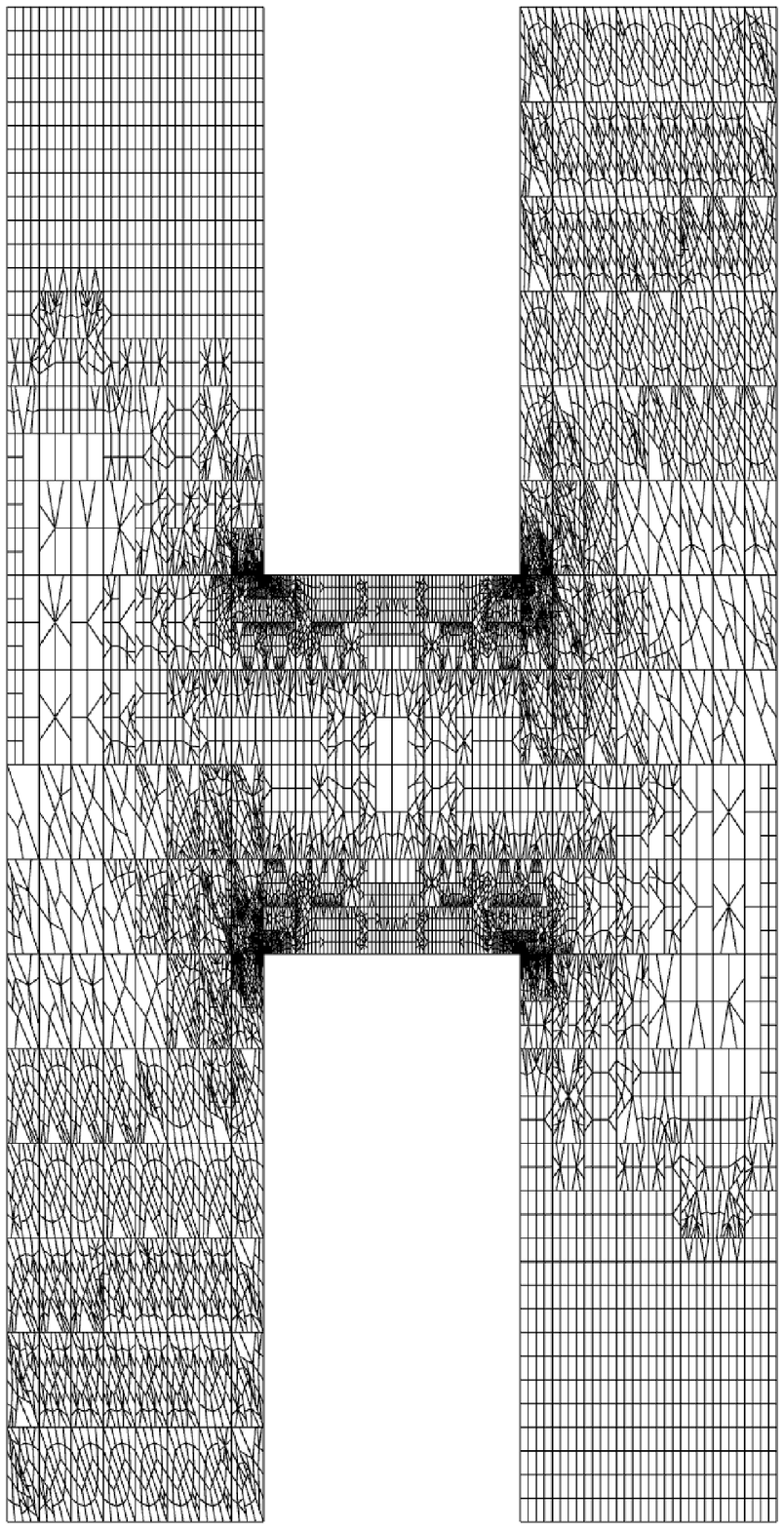}
\end{minipage}
\caption{Adaptively refined meshes obtained with VEM scheme at refinement steps 0, 1 and 8 (Adaptive VEM S-T).}
\label{FIG:VEMHST}
\end{center}
\end{figure}
\begin{figure}[H]
\begin{center}
\begin{minipage}{4.2cm}
\centering\includegraphics[height=6.1cm, width=4.1cm]{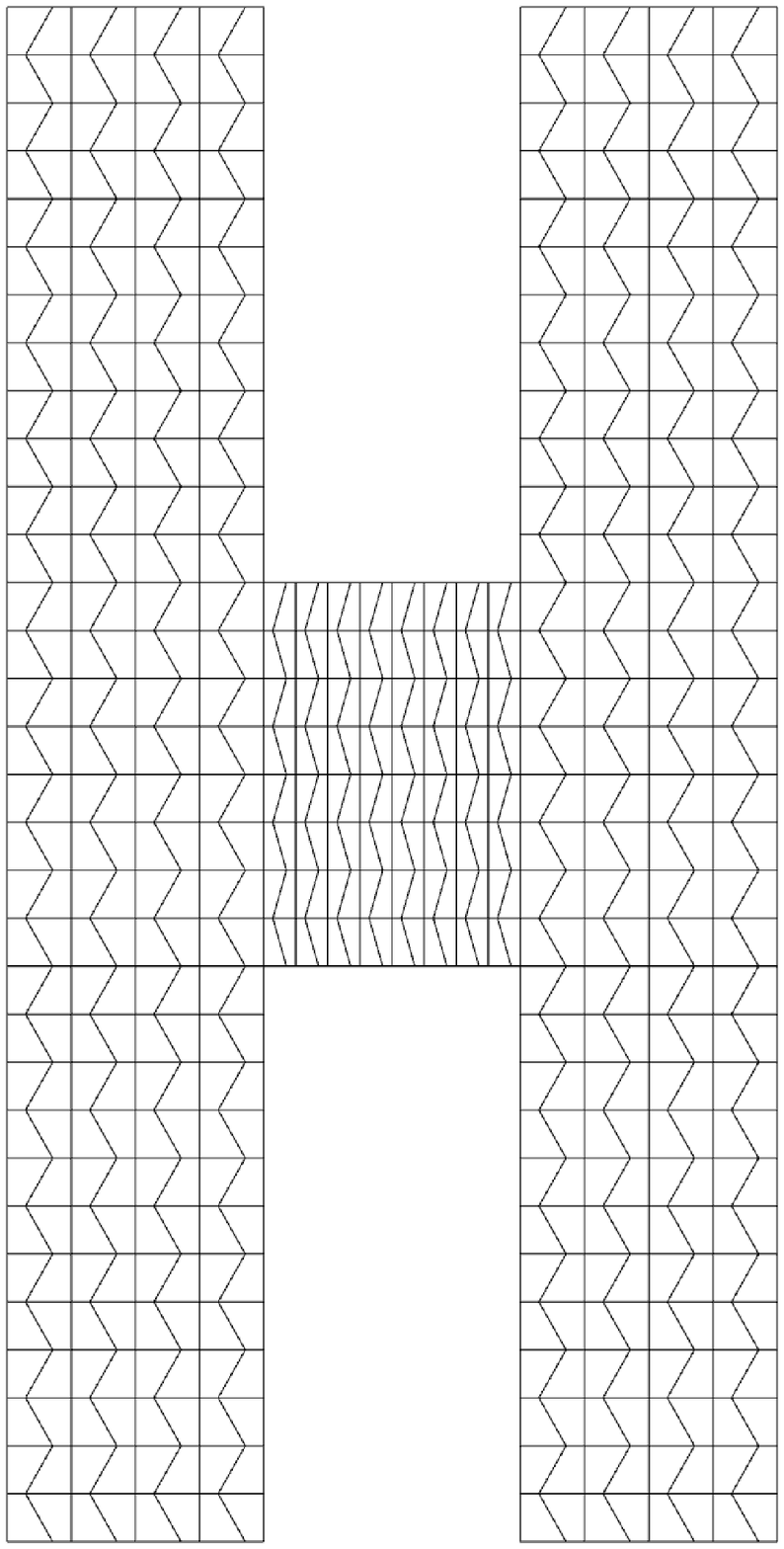}
\end{minipage}
\begin{minipage}{4.0cm}
\centering\includegraphics[height=6.1cm, width=4.1cm]{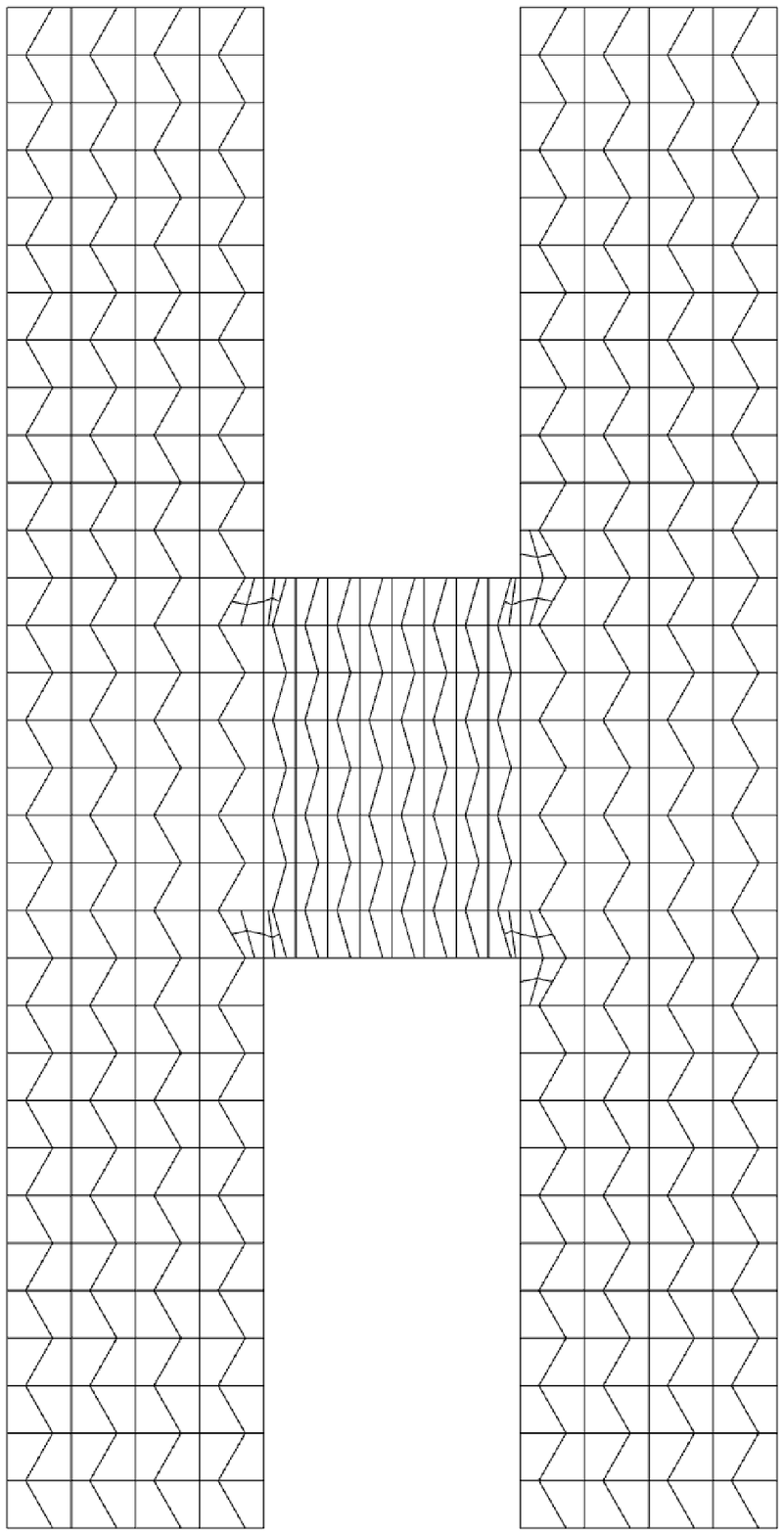}
\end{minipage}
\begin{minipage}{4.0cm}
\centering\includegraphics[height=6.1cm, width=4.1cm]{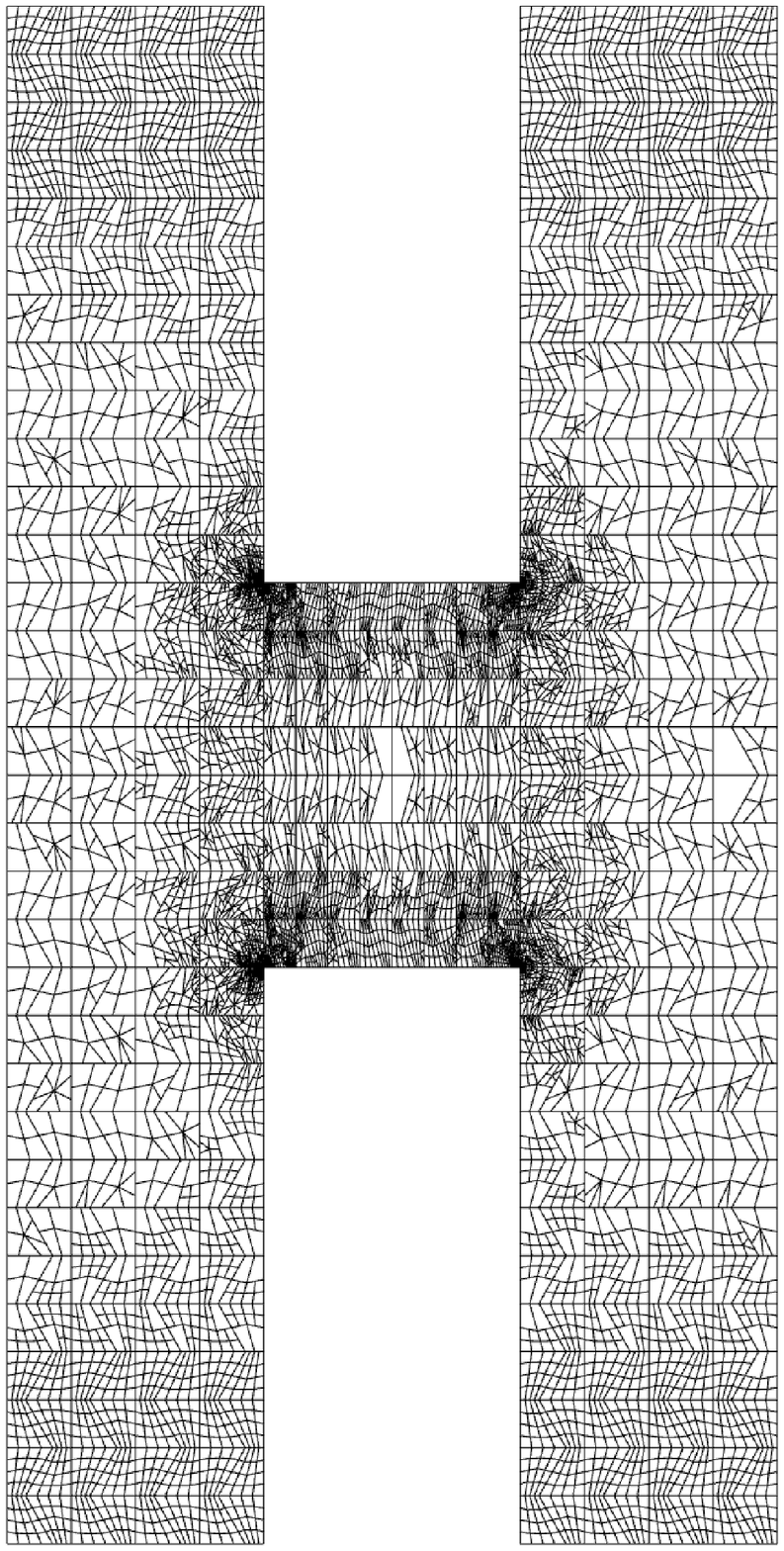}
\end{minipage}
\caption{Adaptively refined meshes obtained with VEM scheme at refinement steps 0, 1 and 8 (Adaptive VEM B).}
\label{FIG:VEMHB}
\end{center}
\end{figure}
\begin{figure}[ht]
\begin{center}
\begin{minipage}{4.2cm}
\centering\includegraphics[height=6.1cm, width=4.1cm]{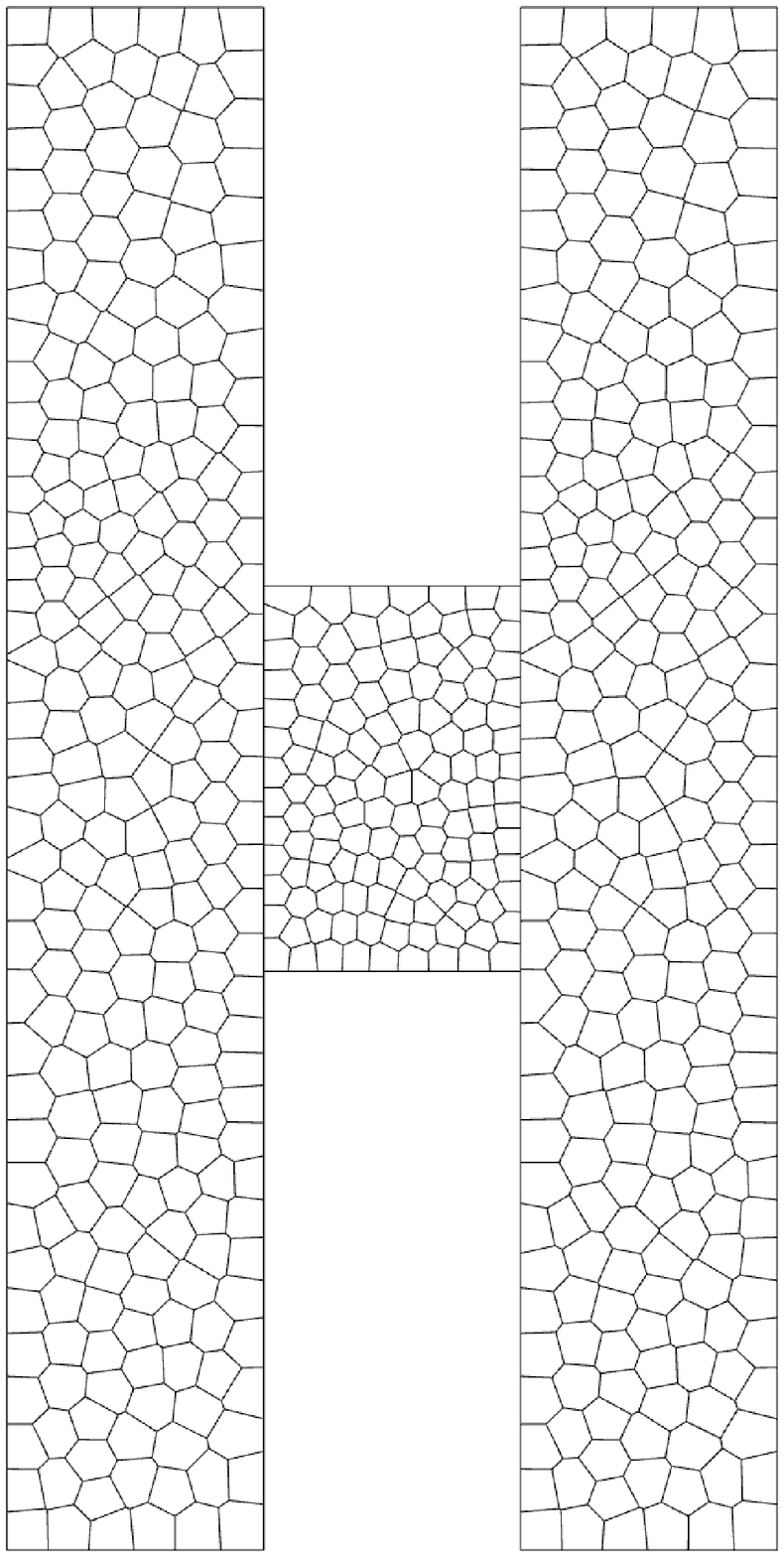}
\end{minipage}
\begin{minipage}{4.0cm}
\centering\includegraphics[height=6.1cm, width=4.1cm]{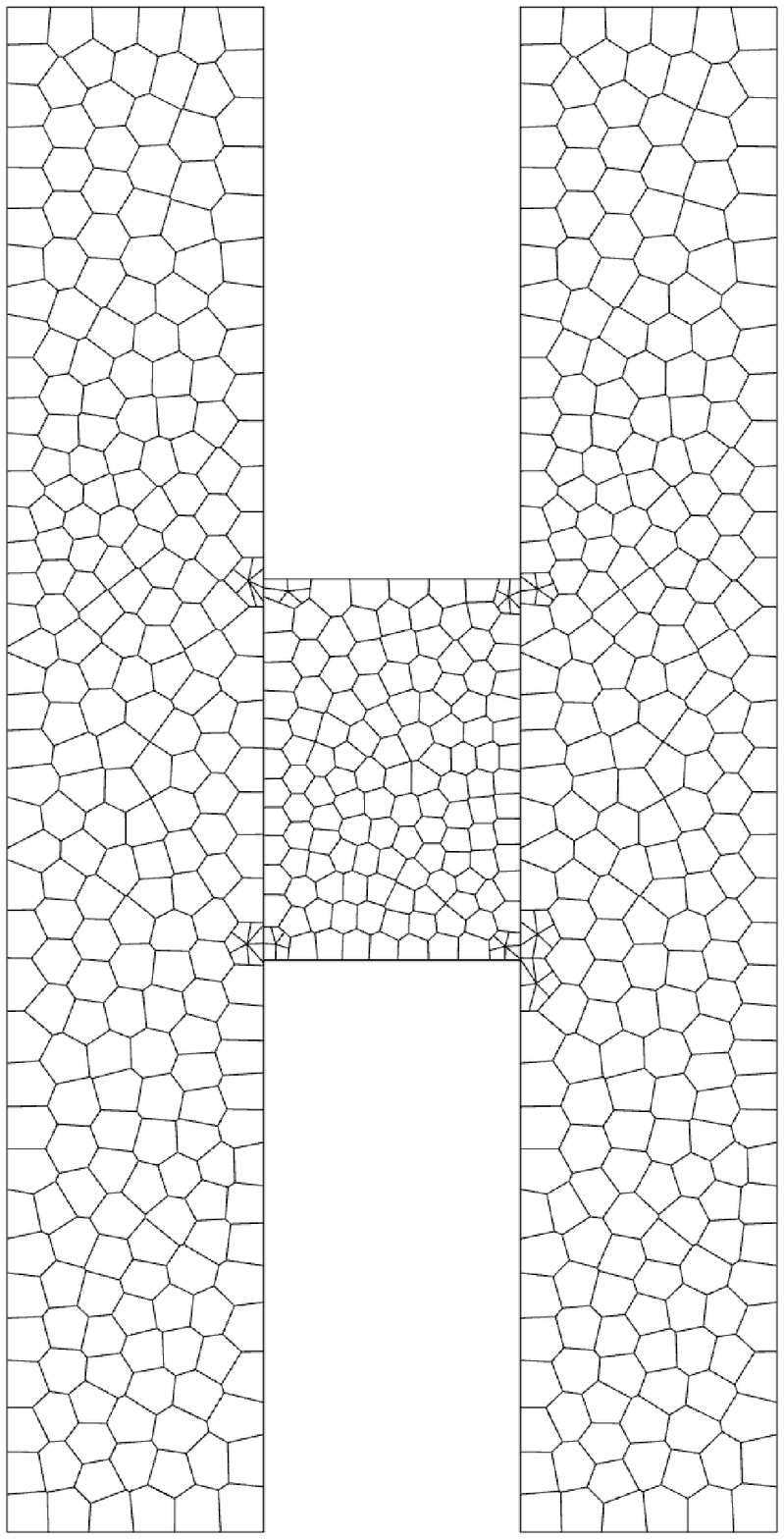}
\end{minipage}
\begin{minipage}{4.0cm}
\centering\includegraphics[height=6.1cm, width=4.1cm]{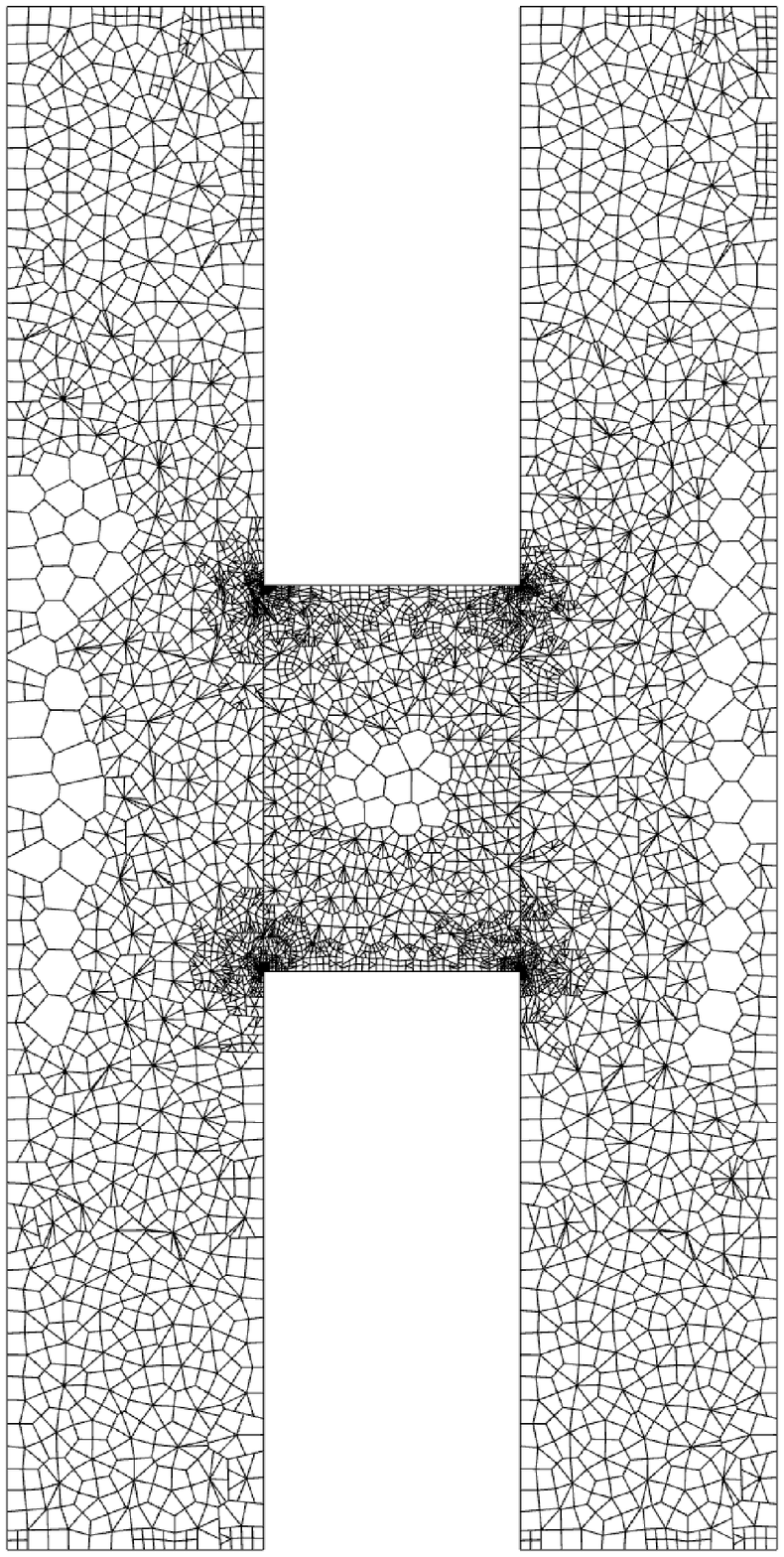}
\end{minipage}
\caption{Adaptively refined meshes obtained with VEM scheme at refinement steps 0, 1 and 8 (Adaptive VEM V).}
\label{FIG:VEMHV}
\end{center}
\end{figure}
Similarly to Test 1, the computations of  the errors $|\lambda_2 -\lambda_{h2}|$, have been obtained with a least squares fitting of the calculated values obtained with extremely refined meshes. Thus, we have obtained the value $\lambda_2= 1.2040$, which has at least four exact significant digits.

In Table \ref{TABLA:3} we report the second lowest eigenvalue $\lambda_{h2}$ on uniformly refined meshes, adaptively refined meshes with different type of initial meshes. Each table includes the estimated convergence rate.

\begin{table}[H]
\begin{center}
\caption{Computed lowest eigenvalue   $\l_{h2}$ computed with different initial meshes.}
\begin{tabular}{|c|c||c|c||c|c||c|c||c|c||}
  \hline
    \multicolumn{2}{|c||}{Uniform VEM}&  \multicolumn{2}{c||}{Adaptive VEM S-T} &  \multicolumn{2}{c||}{Adaptive VEM B} &  \multicolumn{2}{c||}{Adaptive VEM V} \\
    \hline
     $N$ & $\l_{h2}$  &         $N$ & $\l_{h2}$&       $N$ & $\l_{h2}$ &       $N$ & $\l_{h2}$   \\
\hline
916   &1.1831 &     756   &1.1821&1368 &1.1925 &1925 & 1.1959\\
   3560&   1.1960&    832&   1.1893&      1442 &1.1957& 2027 & 1.1981\\
   14032  & 1.2009  &   1068 &  1.1959&     1594& 1.1979& 2152 &1.1992\\
   55712  & 1.2028  &    1830 &  1.1992&    1848& 1.1991& 2593& 1.2003\\
   222016  & 1.2035  &   3304 &  1.2020&    2928& 1.2006& 3588& 1.2013\\
                     &    & 4992   &1.2027&    5564& 1.2024& 6903 &1.2028\\
                     &    &    8130  & 1.2031 &    8093& 1.2028& 9480& 1.2031\\
                     &     & 16320 &  1.2036&   12045& 1.2030& 12846 & 1.2032\\
                     &     &  23706 &  1.2037&   21613 &1.2036& 15214& 1.2033  \\                              \hline 
     Order   &$\mathcal{O}\left(N^{-0.70}\right)$ &   Order   & $\mathcal{O}\left(N^{-1.19}\right)$ &   Order   & $\mathcal{O}\left(N^{-1.09}\right)$ &   Order   & $\mathcal{O}\left(N^{-1.11}\right)$ \\
       \hline
       $\l_2$  &1.2040 &  $\l_{2}$  &1.2040&   $\l_{2}$  &1.2040&   $\l_{2}$  &1.2040\\
     \hline
    \end{tabular}
\label{TABLA:3}
\end{center}
\end{table}    
In Figure \ref{errorH} we present error curves where we observe that the three refinement schemes lead to a correct convergence rate.
It can be seen from Table \ref{TABLA:3} and Figure \ref{errorH}, that the uniform refinement leads to a convergence rate close to that predicted by the theory, while the adaptive VEM  schemes allow us to recover the optimal order of convergence $\mathcal{O}\left(N^{-1}\right)$.

 \begin{figure}[ht]
\begin{center}\begin{minipage}{8.0cm}
\centering\includegraphics[height=8.0cm, width=8.0cm]{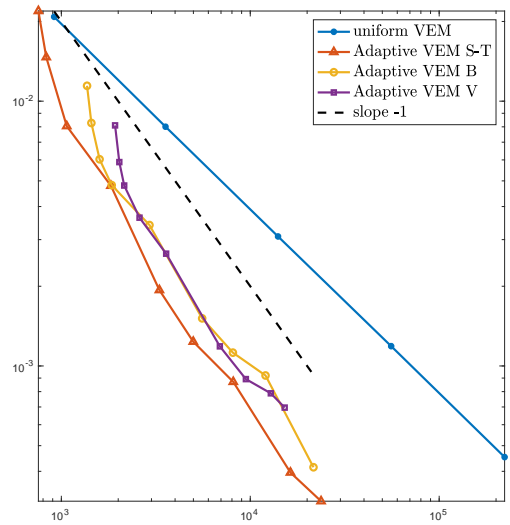}
\end{minipage}
\caption{Error curves of $|\l_{2}-\l_{h2}|$ for uniformly refined meshes
and adaptively refined meshes VEM with different initial meshes.}
\label{errorH}
\end{center}
\end{figure}

In Figure \ref{FIG:eigenfunctionH} we present  plots of the computed eigenfunctions  $\bw_{h2}$ (displacement field) and $p_{h2}$ (pressure fluctuation) corresponding to the second eigenvalue.

\begin{figure}[H]
\begin{center}
\begin{minipage}{4.2cm}
\centering\includegraphics[height=6.1cm, width=4.1cm]{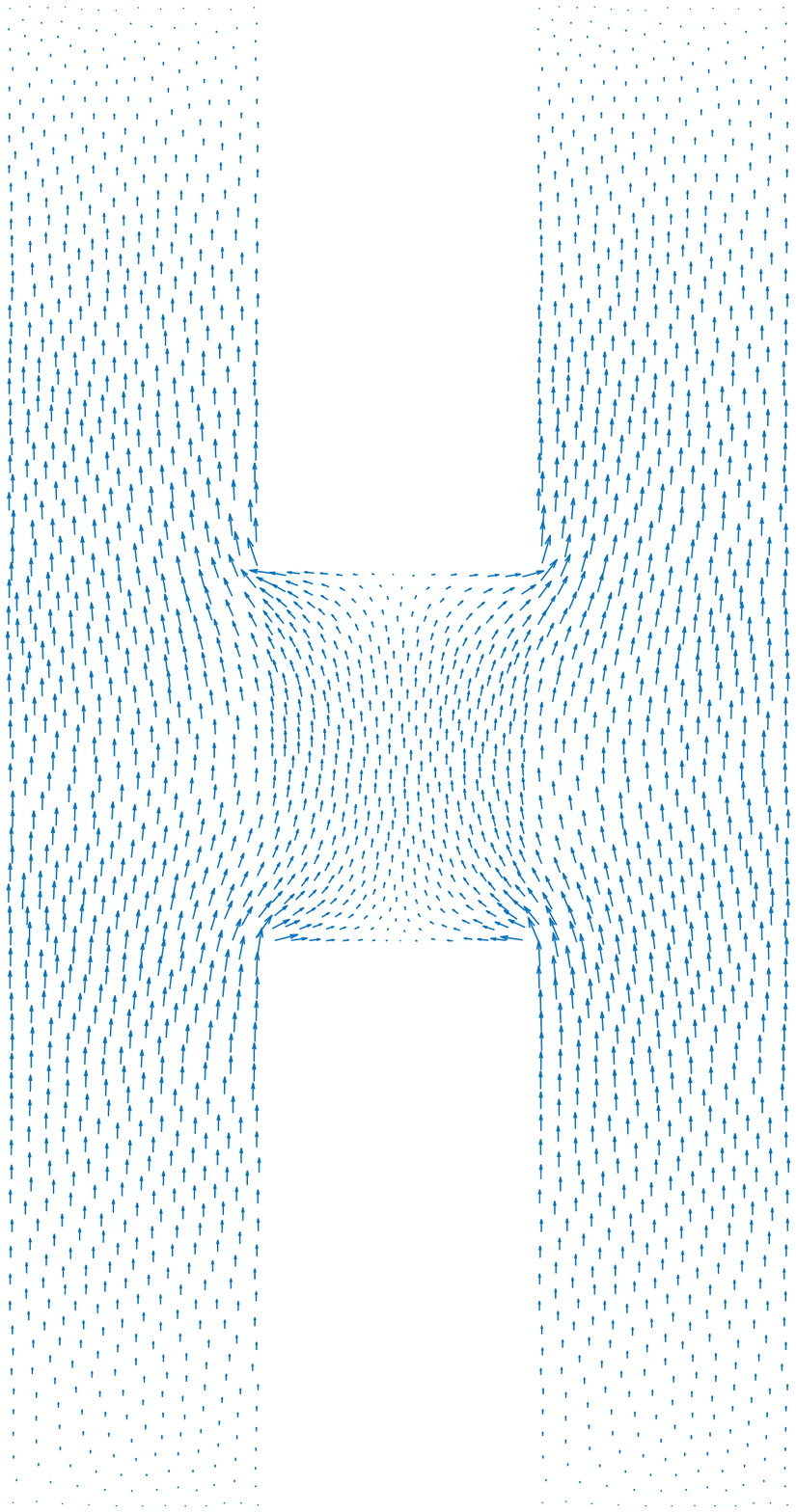}
\end{minipage}
\begin{minipage}{4.2cm}
\centering\includegraphics[height=6.1cm, width=4.1cm]{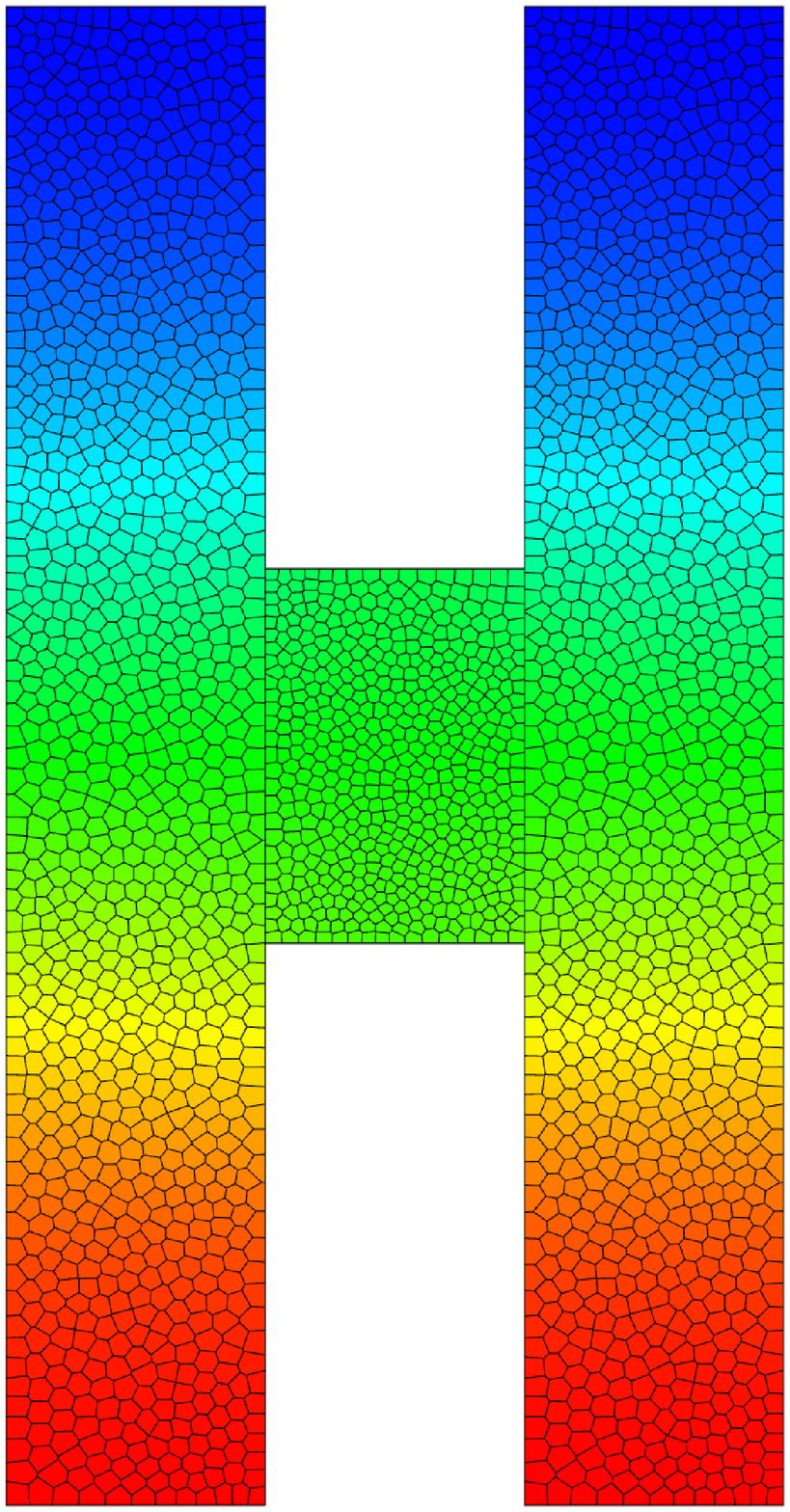}
\end{minipage}
\caption{Test 2. Eigenfunctions of the acoustic problem corresponding to the second lowest eigenvalue: displacement field $\bw_{h2}$ (left), pressure fluctuation $p_{h2}$ (right).}
\label{FIG:eigenfunctionH}
\end{center}
\end{figure}

\subsection{Test 3: Circular domain with obstacles.}
As a third test, we have considered a configuration closer to a real application: four square tubes immersed in a fluid occupying a circular cavity. Clearly in this test there are two relevant geometrical issues: in one hand, we have a non polygonal domain for which we are making an approximation by means of polygonal meshes, and the four rigid squares that lie in the interior of the circle. These tubes lead to non smooth eigenfunctions when the solutions for the acoustic problem are approximated, due the singularities of the corner on each square.

To make matters precise, let us define the circular domain by $\O_C:=\{(x,y)\in\mathbb{R}^2\,:\,x^2+y^2<1\}$ and the squares 
$\O_{\text{I}}:=[1/5,3/5]\times[1/5,3/5]$,  $\O_{\text{II}}:=[-3/5,-1/5]\times [1/5,3/5]$, $\O_{\text{III}}:=[-3/5,-1/5]\times [-3/5,-1/5]$ 
and 
$\O_{\text{IV}}:=[1/5,3/5]\times [-3/5,-1/5]$. 
Hence, the computational domain is $\O:=\O_C\setminus \{ \O_{\text{I}}\cup\O_{\text{II}}\cup\O_{\text{III}}\cup\O_{\text{IV}}\}$. 



In the sequel, we consider the  fourth eigenfunction. In Figure \ref{FIG:VEMdona}  we present an adaptive refinement of our estimator when Voronoi meshes are considered.  On the left hand side we present the initial mesh and,  after 1 and 8 iterations of our numerical method, we observe that the estimator $\boldsymbol{\eta}$ identifies the singularities on the geometry that cause the poor regularity of the eigenfunction, and starts the refinement around these corners in order to recover the optimal order of convergence. 

\begin{figure}[H]
\begin{center}
\begin{minipage}{4.2cm}
\centering\includegraphics[height=4.1cm, width=4.1cm]{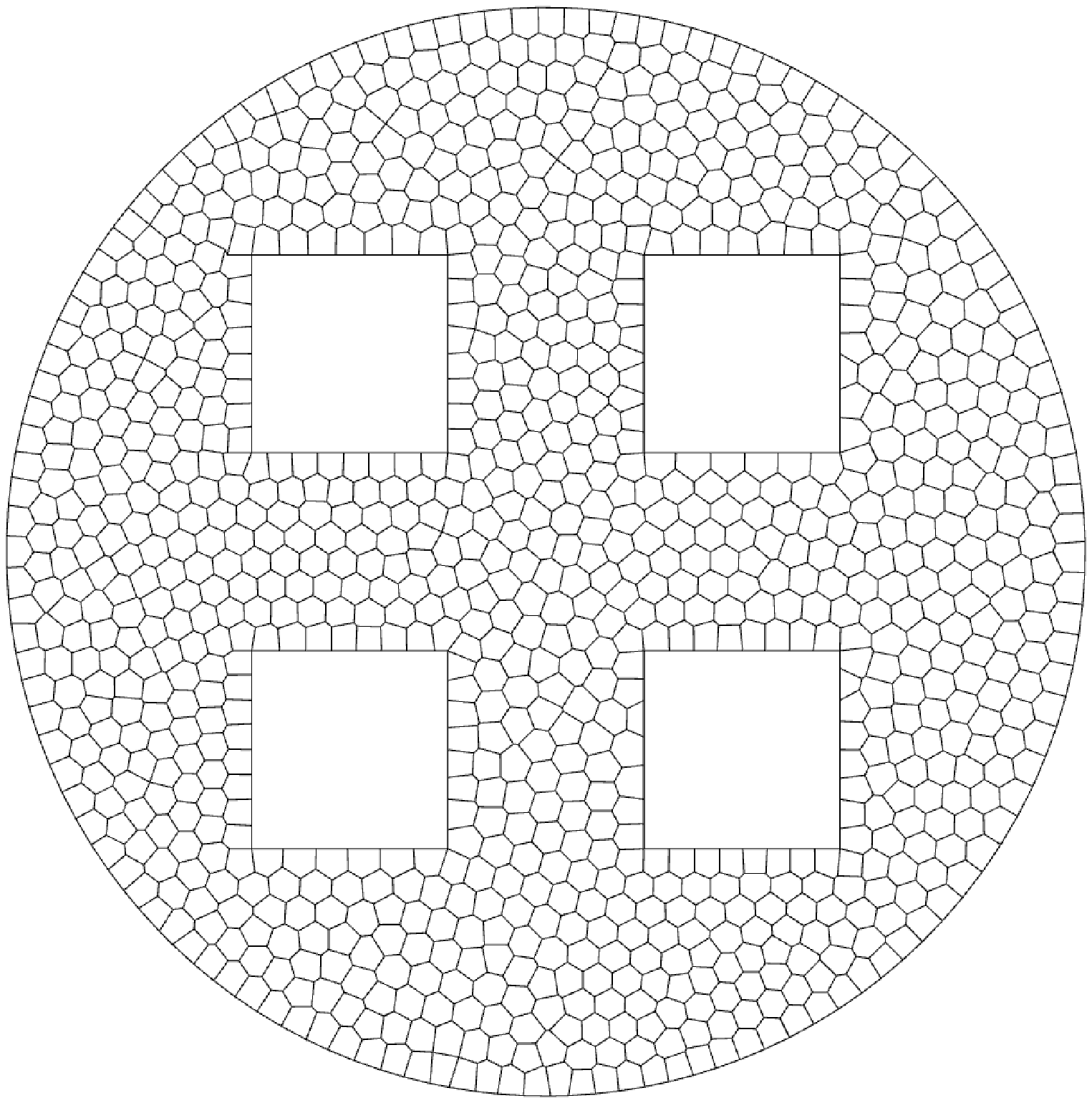}
\end{minipage}
\begin{minipage}{4.0cm}
\centering\includegraphics[height=4.1cm, width=4.1cm]{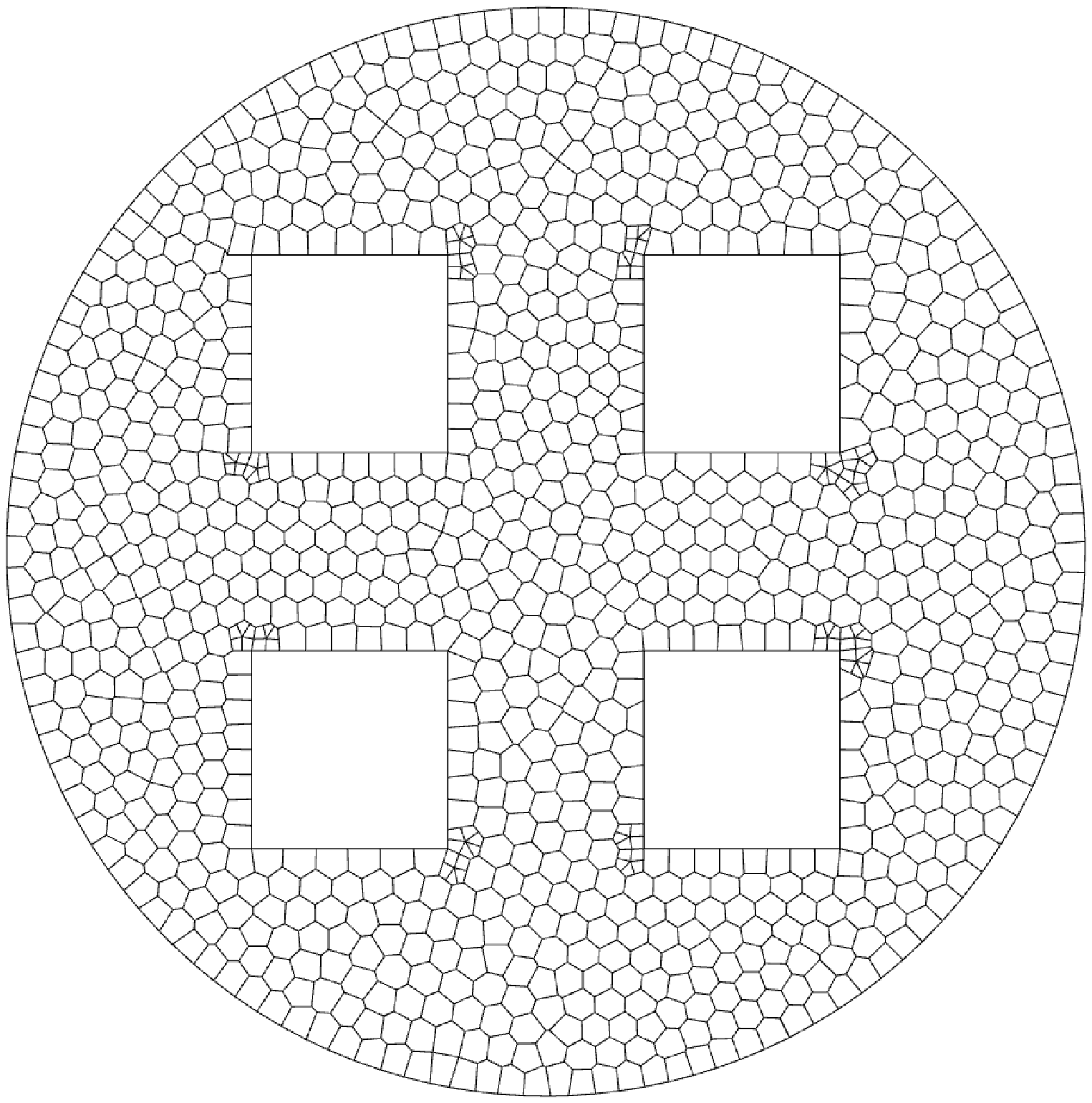}
\end{minipage}
\begin{minipage}{4.0cm}
\centering\includegraphics[height=4.1cm, width=4.1cm]{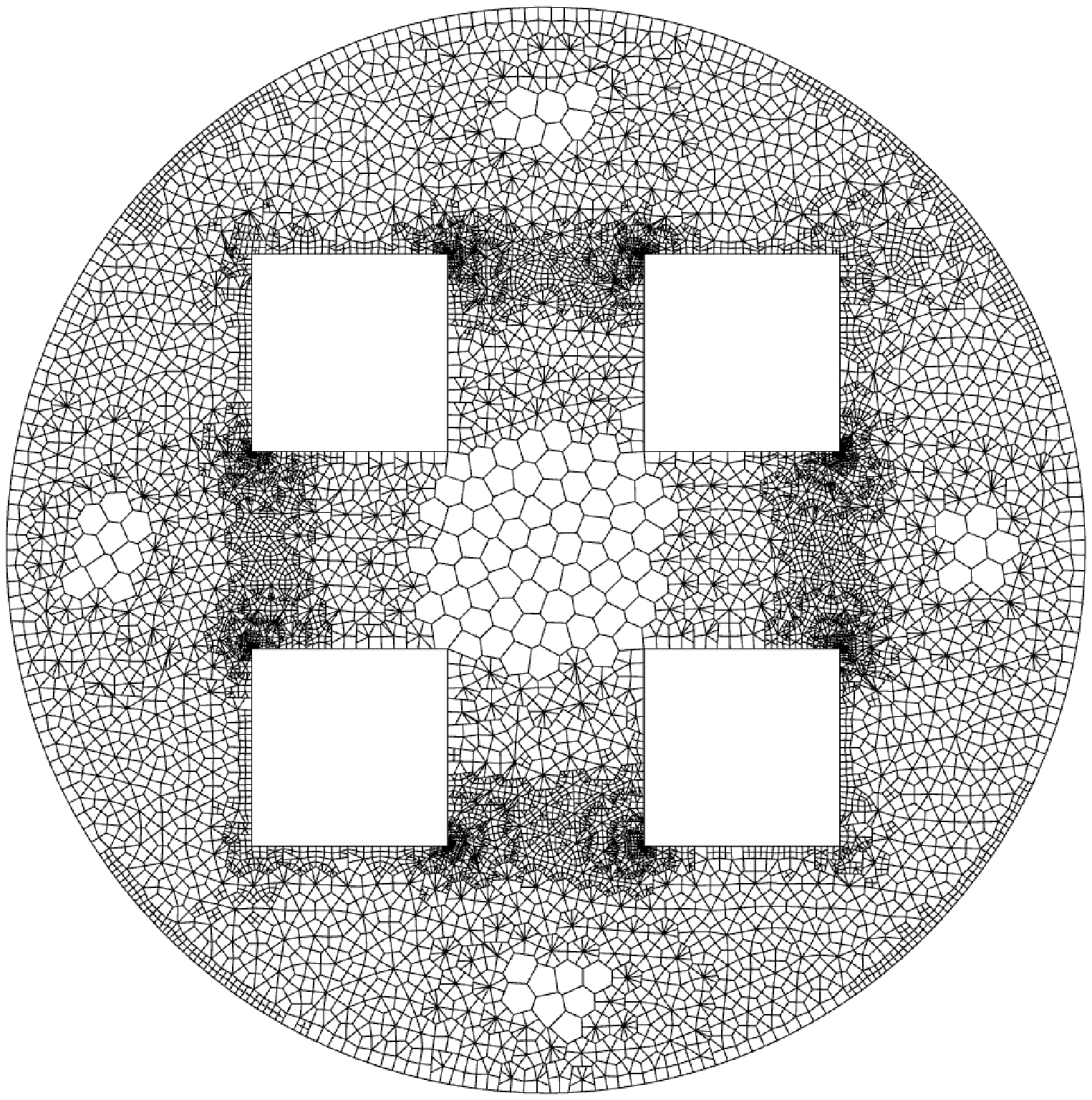}
\end{minipage}
\caption{Adaptively refined meshes obtained with VEM scheme at refinement steps 0, 1 and 8 (Adaptive VEM V).}
\label{FIG:VEMdona}
\end{center}
\end{figure}


Figure~\ref{errorV} shows a logarithmic plot of the errors between the calculated approximations of the fourth smallest positive eigenvalue and the ``exact" one, versus the number of degrees of freedom N of the meshes. As in the previous two tests, the exact value of the fourth eigenvalue is obtained by using a least squares fit. The figure shows the results obtained with "uniform" meshes and with adaptively refined meshes and shows how the optimal order of convergence is recovered. Finally, Figure~\ref{FIG:eigenfunctiondona} shows
the eigenfunctions of the acoustic problem corresponding to the fourth lowest eigenvalue.
 \begin{figure}[H]
\begin{center}\begin{minipage}{8.0cm}
\centering\includegraphics[height=8.0cm, width=8.0cm]{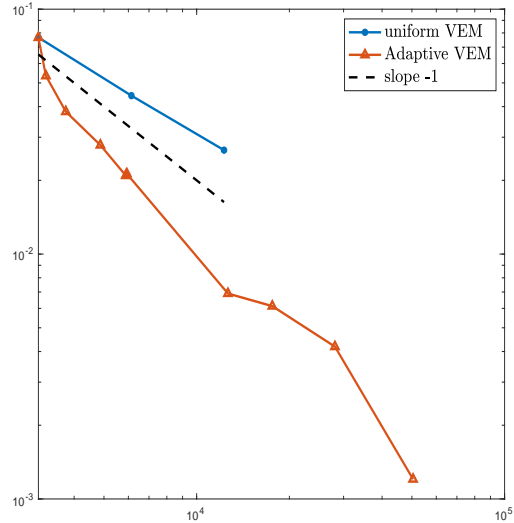}
\end{minipage}
\caption{Error curves of $|\l_{4}-\l_{h4}|$ for uniformly refined meshes
and adaptively refined meshes VEM.}
\label{errorV}
\end{center}
\end{figure}

\begin{figure}[H]
\begin{center}
\begin{minipage}{6.2cm}
\centering\includegraphics[height=6.1cm, width=6.1cm]{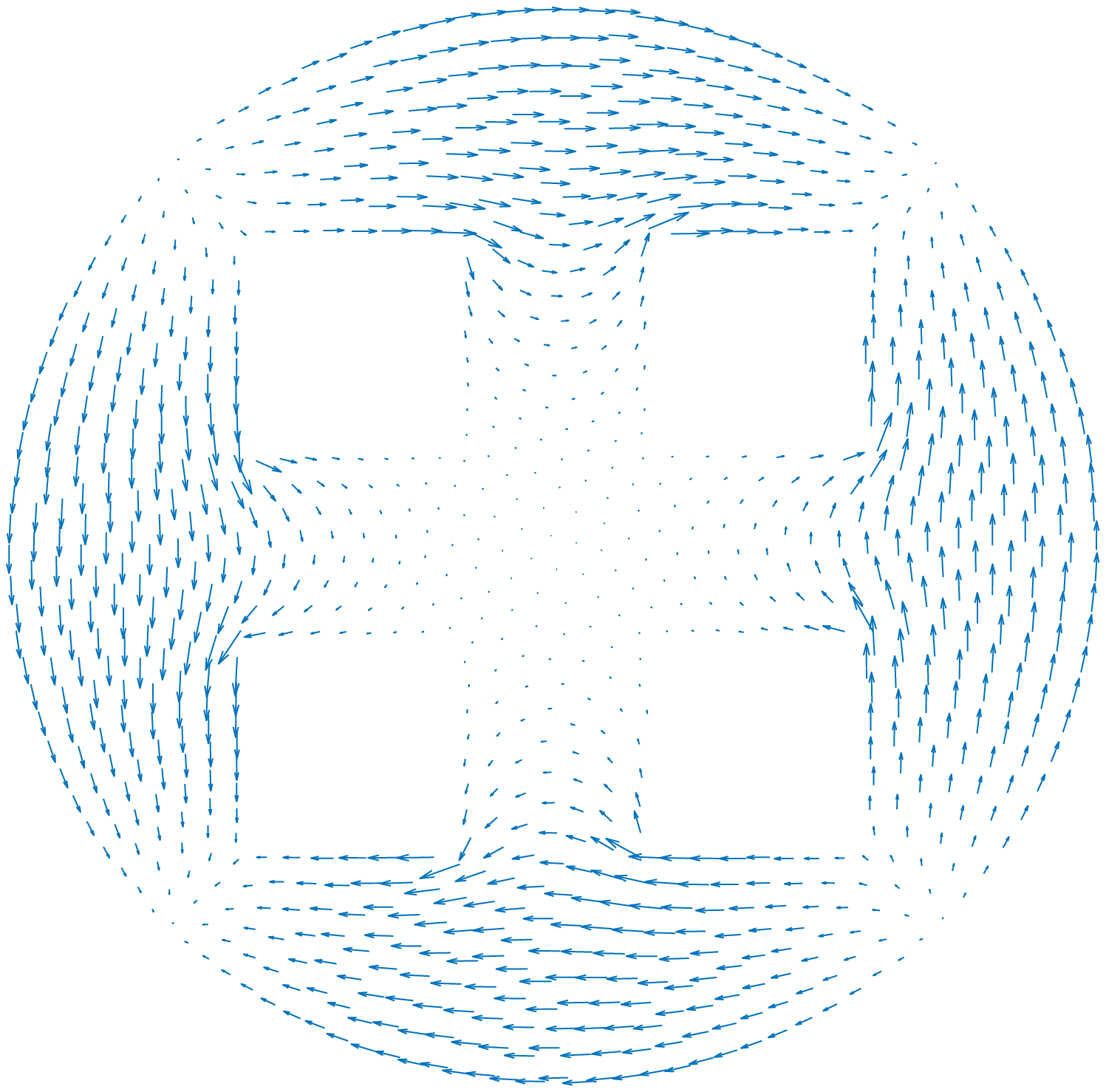}
\end{minipage}
\begin{minipage}{6.0cm}
\centering\includegraphics[height=6.1cm, width=6.1cm]{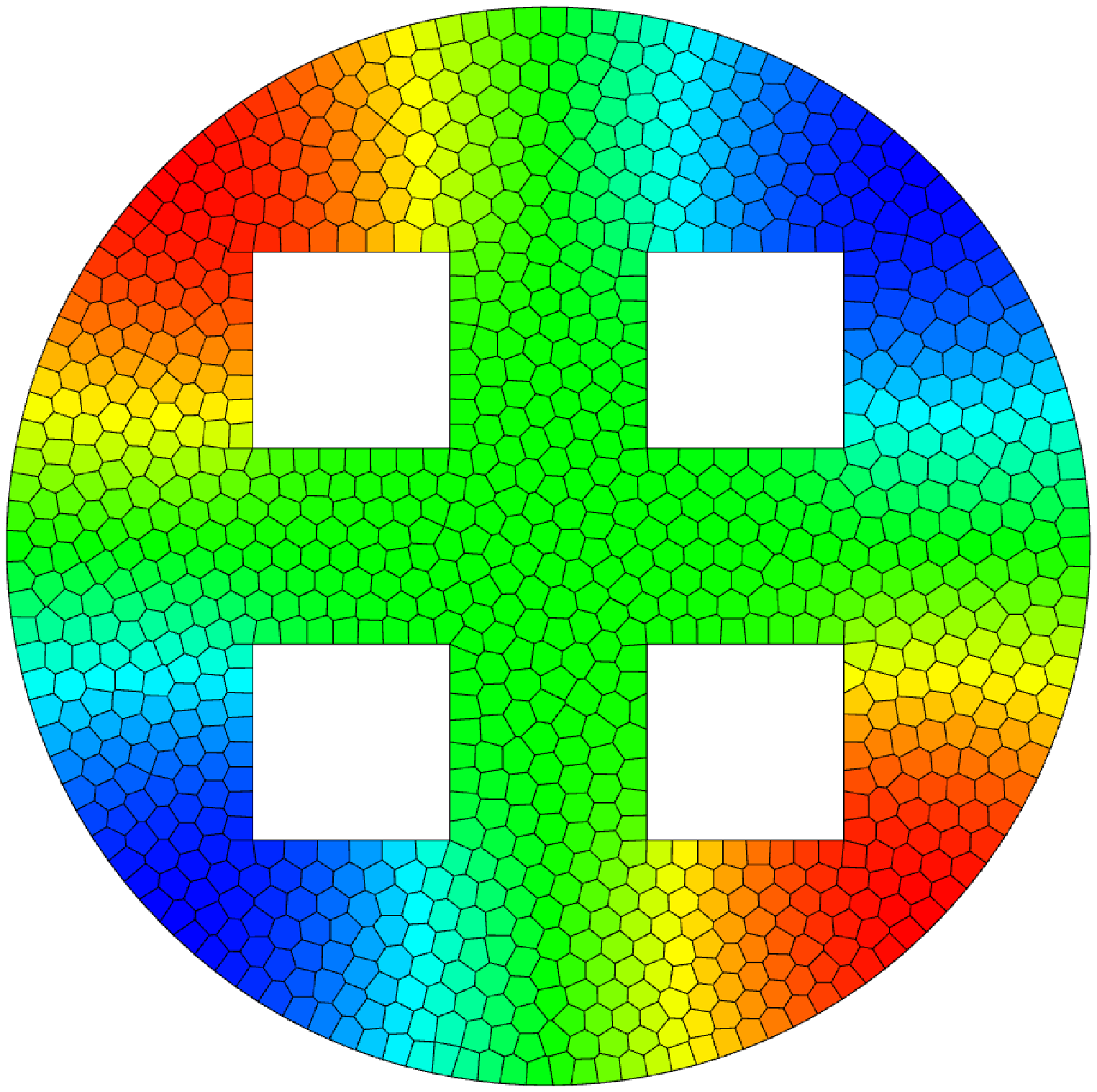}
\end{minipage}
\caption{Test 3. Eigenfunctions of the acoustic problem corresponding to the second lowest eigenvalue: displacement field $\bw_{h4}$ (left), pressure fluctuation $p_{h4}$ (right).}
\label{FIG:eigenfunctiondona}
\end{center}
\end{figure}

\section{Conclusions}
In this work, we have derived and analyzed an a posteriori error estimate for the acoustic vibration problem by means of mixed virtual element discretization. The theoretical analysis developed in this work was strongly supported by superconvergence results for mixed spectral formulations. Several numerical tests that substantiate the theoretical results were presented, confirming that the proposed estimator is capable of recover the optimal order of convergence, as theory predicts.  Moreover, we stress that the present analysis can be extended to the tridimensional case by using the VEM spaces introduced in \cite{BBMR-NM2016} and the recent results for interpolation estimates derived in \cite{BMMArXiv22}.

\bibliographystyle{amsplain}

\end{document}